\documentclass[10pt,reqno]{amsart}

\usepackage[a4paper, margin=1.5in]{geometry}

\usepackage {color}
\usepackage{amsmath,amsthm,amssymb}
\usepackage{epsfig}
\usepackage{graphicx}
\usepackage{amssymb}
\usepackage{latexsym}
\usepackage{pdfpages}

\usepackage{pgf}

\usepackage{ulem} 

\usepackage{ifpdf}

\def\z{{\bf z}}

\ifpdf 
  \usepackage[hidelinks]{hyperref}
\else 
\fi 

 \usepackage[usenames,dvipsnames]{pstricks}
 \usepackage{epsfig}
 \usepackage{pst-grad} 
 \usepackage{pst-plot} 

\newtheorem{theorem}{Theorem}[section]
\newtheorem{lemma}[theorem]{Lemma}
\newtheorem{definition}[theorem]{Definition}
\newtheorem{proposition}[theorem]{Proposition}

\newtheorem{remark}[theorem]{Remark}
\newtheorem{example}[theorem]{Example}

\newtheorem*{theorem*}{\it Theorem}

\def\vint_#1{\mathchoice%
          {\mathop{\kern 0.2em\vrule width 0.6em height 0.69678ex depth -0.58065ex
                  \kern -0.8em \intop}\nolimits_{\kern -0.4em#1}}%
          {\mathop{\kern 0.1em\vrule width 0.5em height 0.69678ex depth -0.60387ex
                  \kern -0.6em \intop}\nolimits_{#1}}%
          {\mathop{\kern 0.1em\vrule width 0.5em height 0.69678ex
              depth -0.60387ex
                  \kern -0.6em \intop}\nolimits_{#1}}%
          {\mathop{\kern 0.1em\vrule width 0.5em height 0.69678ex depth -0.60387ex
                  \kern -0.6em \intop}\nolimits_{#1}}}
\def\vintslides_#1{\mathchoice%
          {\mathop{\kern 0.1em\vrule width 0.5em height 0.697ex depth -0.581ex
                  \kern -0.6em \intop}\nolimits_{\kern -0.4em#1}}%
          {\mathop{\kern 0.1em\vrule width 0.3em height 0.697ex depth -0.604ex
                  \kern -0.4em \intop}\nolimits_{#1}}%
          {\mathop{\kern 0.1em\vrule width 0.3em height 0.697ex depth -0.604ex
                  \kern -0.4em \intop}\nolimits_{#1}}%
          {\mathop{\kern 0.1em\vrule width 0.3em height 0.697ex depth -0.604ex
                  \kern -0.4em \intop}\nolimits_{#1}}}

\def\R{\mathbb R}

\numberwithin{equation}{section}

\def\1{\raisebox{2pt}{\rm{$\chi$}}}

\newcommand{\Z}{{\mathbb Z}}

\definecolor{violet(ryb)}{rgb}{0.53, 0.0, 0.69}
\usepackage{mathtools}
\mathtoolsset{showonlyrefs}

\begin{document}

\title[Two for sandpiles in weighted graphs]
{\bf  Two models for
sandpile growth  in  weighted graphs}

\author{ J. M. Maz\'{o}n and J. Toledo}

\address{ Jos\'{e} M. Maz\'{o}n
\hfill\break\indent Departament d'An\`{a}lisi Matem\`atica,
Universitat de Val\`encia \hfill\break\indent Valencia, Spain.}
\email{{\tt  mazon@uv.es}}

\address{ Julian Toledo
\hfill\break\indent Departament d'An\`{a}lisi Matem\`atica,
Universitat de Val\`encia \hfill\break\indent Valencia, Spain.}
\email{{\tt toledojj@uv.es}}

\date{}

\keywords{Weighted graphs, $p$-Laplacian, $\infty-$Laplacian, sandpiles, mass transport theory.\\
\indent 2020 {\it Mathematics Subject Classification.} 35R02, 47H05, 47H06, 35B40.}


\begin{abstract}
In this paper we study   $\infty$-Laplacian type diffusion equations in weighted graphs obtained
as  limit as $p\to \infty$  to two types of $p$-Laplacian evolution equations  in such graphs. We propose these diffusion equations, that are governed by the subdifferential of a convex energy functionals associated to the indicator function of the set
$$K^G_{\infty}:= \left\{ u \in  L^2(V, \nu_G) \ : \ \vert u(y) - u(x) \vert \leq 1 \ \ \hbox{if} \ \ x \sim y  \right\}$$
and the set
$$K^w_{\infty}:= \left\{ u \in  L^2(V, \nu_G) \ : \ \vert u(y) - u(x) \vert \leq \sqrt{w_{xy}} \ \ \hbox{if} \ \ x \sim y  \right\}$$
as   models for  sandpile growth in weighted graphs.
Moreover, we also analyse the collapse of the initial condition
when it does not belong to  the stable sets $K^G_{\infty}$ or $K^w_{\infty}$ by means of an abstract result given in~\cite{BEG}.
We give an interpretation of the limit problems in terms
of Monge-Kantorovich mass transport theory. Finally, we give some explicit solutions of simple examples that illustrate the dynamics of the sandpile growing or collapsing.

\end{abstract}

\maketitle

%

\section{Introduction}

Our aim is to study the limit, as $p\to \infty$, of $p$-Laplacian evolution problems in the framework of the weighted graphs
and to interpret and propose the limit problems,  a sort of $\infty$-Laplacian type diffusion problems, as sandpile models in weighted graphs. 
This proposal is inspired by the  model proposed by  Evans and Rezakhanlou in \cite{ER} which is formulate in the lattice $\Z^n \subset \R^n$ and describes a kind of stochastic  microscopic particle model for the macroscopic sandpile dynamics introduced by  Prigozhin \cite{P1}, \cite{P2}, by Aronsson, Evans and  Wu \cite {AEW} and by Evans, Feldman and Gariepy \cite{EFG}.

 The physicists Bak, Tang, and Wiesenfeld \cite{BTW} created an idealized version of a sandpile in which sand is
stacked on the vertices of a graph and is subjected to certain avalanching rules. They used the model as
an example of what they called self-organized criticality. The abelian sandpile model is a variation, due
to the physicist Deepak Dhar in 1990 \cite{Dhar}, in which the avalanching obeys a useful commutativity rule.
There is a abundant literature on sandpile models in discrete graphs much of it relating to the abelian sandpile model or chip-firing game model (see for instance \cite{J}, \cite{JP}, \cite{R}) but in this case the discrete graphs are regular, that is all the weights are equal, in many case the graphs $\Z^N$ are considered. Here we consider  different sandpile models and in general weighted graphs in which the   weights are relevant.

 We now recall some results from~\cite{AEW} and~\cite{EFG} (see also~\cite{Evans}) since we adapt the same procedure for the results presented in the framework of weighted graphs. In such references it was
investigated the limiting behaviour as $p \to \infty$ of the solutions
to the quasilinear parabolic problem
\begin{equation}\label{p_Lap1.intro}
P_{p}(u_0,f)\quad\left\{\begin{array}{ll} v_{p,t}- \Delta_p v_p = f
  &\hbox{in } ]0,T[\times\R^N,
\\[10pt]
v_p(0,x)=u_0(x)     &\hbox{in } \R^N,
\end{array}\right.
\end{equation}
where $\Delta_p u = {\rm div} \, (\vert \nabla u \vert^{p-2} \nabla u )$ and   $f$ is a nonnegative function which
is interpreted physically as  a source term that adds material to the evolving
system  within which mass particles are continually rearranged by
diffusion. By considering the functional
\begin{equation}\label{e2Mazo.intro}
F_p(v) = \left\{ \begin{array}{ll} \displaystyle\frac{1}{p} \int_{\R^N} |\nabla v(y)|^{p}
\, dy     & \hbox{if} \   u
\in L^2(\R^N) \cap W^{1,p}(\R^N), \\[10pt]
+ \infty   &
\hbox{if} \   u\in L^2(\R^N) \setminus  W^{1,p}(\R^N), \end{array}
\right.
\end{equation}
$1<p<\infty$,  the PDE problem $P_{p}(u_0,f)$ has the standard
reinterpretation
$$
\left\{\begin{array}{ll} f(t) - v_{p,t} = \partial F_p( v_p(t)),    &\hbox{ a.e.}   \ t \in
]0,T[,
\\[10pt]
v_p(0,x)=u_0(x)    &\hbox{ in } \R^N.
\end{array}\right.
$$
In \cite{AEW}, assuming that $u_0$ is a Lipschitz function with compact support,
satisfying
$$ \Vert \nabla u_0 \Vert_{\infty} \leq 1,$$
and for $f$ a smooth nonnegative function with compact support in
$[0,T]
\times \R^N$, it was proved that there exists a sequence $p_i \to
+ \infty$ and   a limit function $v_{\infty}$  such that, for
each $T > 0$,
$$
\left\{\begin{array}{ll}  v_{p_i} \to v_{\infty}, \quad &\hbox{ a.e. and in} \ L^2(\R^N \times
]0,T[),
\\[10pt]
\nabla v_{p_i} \rightharpoonup \nabla v_{\infty}, \ v_{p_i,t}\rightharpoonup
v_{\infty, t} \quad &\hbox{ weakly in } \ L^2(\R^N \times ]0,T[),
\end{array}\right.
$$
and moreover the limit function $v_{\infty}$ satisfies
$$
P_{\infty}(u_0,f)\quad\left\{\begin{array}{ll} f(t) - v_{\infty,t}
\in
\partial F_\infty ( v_{\infty}(t)), \quad &\hbox{ a.e.} \ \ t \in
]0,T[,
\\[10pt]
v_{\infty}(0,x)=u_0(x), \quad &\hbox{ in } \R^N,
\end{array}\right.
$$
where
$$
F_\infty (v) = \left\{ \begin{array}{ll} 0 \qquad & \hbox{if} \ v\in L^2(\R^N),\ \vert
\nabla
v \vert \leq 1 ,\\[10pt]
+ \infty \qquad & \hbox{in other case}. \end{array}\right.
$$
This limit problem $P_{\infty}(u_0,f)$ was understood as a model that explains the growth of a
sandpile ($v_\infty (t,x)$ describes the amount of the sand at the
point $x$ at time $t$) under the action of the source term $f$,  the main assumption being that the slope of the
sandpile must be less or equal than $1$   ($|Dv_\infty|\le 1$).
In \cite{EFG}  (see also~\cite{BEG}) it was studied the collapsing of the initial
condition phenomena for the local problem $P_{p}(u_0,0)$ when the
initial condition $u_0$ satisfies $\Vert \nabla u_0 \Vert_{\infty}
> 1$. It was proved that the limit of the solutions $v_p (t,x)$ to
$P_{p}(u_0,0)$, as $p\to\infty$, is  an stable configuration  independent of time.
And it was described the small layer in which the solution
rapidly changes from being $u_0$ at an initial time to
the final stationary limit. Similar problems in $\mathbb{R}^N$ under nonlocal diffusion driven by a regular kernel were studied in~\cite{AMRTsp} (see also~\cite{AMRTams}).

A weighted graph is defined as a special type of graph in which the edges are assigned some weights which represent cost, distance, and many other measuring units. Weighted graphs are an accurate representation of many real-world scenarios, where the relationships between entities have varying degrees of importance.
On the other hand, one can find in the literature different definitions of   $p$-Laplacian type operator in weighted graphs.  We focus our attention on two that are typically used.
 For each of these $p$-Laplacian operators
we will study similar problems to~\eqref{p_Lap1.intro}, and take limits as $p\to\infty$ to get different evolution problems  for each of them  that can be seen as sandpile growing models.
 More concretely, consider a connected weighed graph with weights $w_{xy}>0$ between   related vertices $x\sim y$ ($w_{xy}=0$ otherwise) and weighted degree $d_x=\sum_{y\sim x}w_{xy}$ on each vertex $x$ (see more details later on). Starting with  the $p$-Laplacian given by
$$\Delta_p^G u (x) =    \frac{1}{d_x} \sum_{y \in V} \vert \nabla u(x,y) \vert^{p-2} \nabla u(x,y) w_{xy},$$
 we arrive in the limit to the  problem
\begin{equation}\label{NOnLOpiii}%
     \left\{\begin{array}{l} \displaystyle f(t,
\cdot)- u_t
(t,.) \in \partial I_{K^G_{\infty}} (u(t,.)), \ \ \ \hbox{ a.e.} \ t \in ]0,T[, \\[10pt]
u(0,x)=u_0 (x),
\end{array}\right.
\end{equation}
where
$$K^G_{\infty}= \left\{ u \in L^2(V, \nu_G) \ : \ \vert u(y) - u(x) \vert \leq 1 \ \ \hbox{if} \ \ x \sim y  \right\},$$
that can be seen as a sandpile growing model ($u(t,x)$ represents the height of the sand at the vertex $x$ at time $t$) in which the
sandpile is growing on the vertices $x$ where $f(t,.)>0$ while   the   slope constraint condition  $\vert u(t, y) - u(t, x) \vert \leq 1$ if $y\sim x$; now once the slope  condition may be exceeded, the sandpile must growth on $y\sim x$ (in order to preserve such constraint), and with the same argument in other vertex $z\sim y$, and so on. The set $K^G_{\infty}$ is the set of  stable configurations. Here  the weighted degrees play  a role in the growth dynamics.    On the other hand, with the  $p$-Laplacian given by
$$\Delta_p^w u (x) =      \frac{1}{d_x} \sum_{y \in V} \left(\sqrt{w_{xy}}\right)^{p-2}\vert \nabla u(x,y) \vert^{p-2} \nabla u(x,y)w_{xy},$$
 we arrive in the limit to the problem
\begin{equation}\label{NOnLOppesadoiii}
 \left\{\begin{array}{l} \displaystyle f(t,
\cdot)- u_t
(t) \in \partial I_{K^w_{\infty}} (u(t)), \ \ \ \hbox{ a.e.} \ t \in ]0,T[, \\[10pt]
u(0,x)=u_0 (x).
\end{array}\right.
\end{equation}
where
$$K^w_{\infty}= \left\{ u \in L^2(V, \nu_G), \ \vert u(y)-u(x) \vert\leq \frac{1}{\sqrt{w_{xy}}}  \quad \hbox{if} \ x \sim y\right\},$$
that can be seen as model for sandpile growing  in which the
slope constraint is  $\vert u(t, y) - u(t, x) \vert \leq \frac{1}{\sqrt{w_{xy}}}$, so in this case the weights of edges directly play  a role in the dynamics and in stability.

 The slope constraint is the main  factor in the sandpile evolution models proposed. It determines how the configuration $u(t,.)$ at vertices is under the action of a source term.

The $p$-Laplacian evolution problems and their limits as $p\to\infty$ are studied in Sections~\ref{p-LaplacianE}, \ref{sect.covergtop} and \ref{p-LaplacianEpesado}, \ref{sect.covergtoppesado}.  We also study the corresponding collapsing models  under the action of an unstable configuration at Sections~\ref{Sect.collapsing} and ~\ref{Sect.collapsingW}.  We   describe the sandpile models from the point of view of mass transport theory (Sections~\ref{Sect.Mass.Transfer} and~\ref{Sect.Mass.TransferW}). Concrete simple examples are also given in order to illustrate the dynamics involved (Section~\ref{sect.explicit} and~\ref{sect.explicitpesado}).

\section{Preliminaries}\label{sect.prelim}

As for the local case, to identify the limit of the solutions to the $p$-Laplacian evolution problem that we will consider we   use
methods of convex analysis  and nonlinear semigroup theory. So, we  first recall some terminology (see \cite{ET},
\cite{Brezis} and \cite{Attouch}) and introduce known results that we need.

If $H$ is a real Hilbert space with inner product $( \ , \ )$ and
$\Psi : H \rightarrow (- \infty, + \infty]$ is convex, then the
subdifferential of $\Psi$ is defined as the multivalued operator
$\partial \Psi$ given by
$$v \in \partial \Psi(u) \ \iff \ \Psi(w) - \Psi(u) \geq (v, w -u)
\ \ \ \forall \, w \in H.$$

Given $K$  a closed convex subset of $H$, the indicator function of $K$ is defined by
$$
I_K(u) = \left\{ \begin{array}{ll} 0 \qquad & \hbox{if} \ \ u \in K
,\\[10pt]
+ \infty
\qquad & \hbox{if} \ \ u \not\in K. \end{array}\right.
$$
It is easy to see that
\begin{equation} \label{e1Mazo}v \in \partial I_K(u) \
\iff \ u \in K \ \ \hbox{and} \ \ (v, w - u ) \leq 0 \ \ \ \forall \, w \in K.
\end{equation}

 In the case that the convex functional  $\Psi : H \rightarrow (- \infty, + \infty]$ is
 proper, lower semi\-continuous  and $\min \Psi=0$, it is well known (see \cite{Brezis}) that the abstract Cauchy
 problem
 \begin{equation} \label{ACP}
\left\{\begin{array}{ll} u^{\prime}(t) + \partial \Psi (u(t)) \ni
f(t), \quad &\hbox{ a.e} \ t \in ]0,T[,
\\[10pt]
u(0)=u_0,&
\end{array}\right.
\end{equation}
has a unique strong solution for any $f \in L^2(0,T;H)$ and $u_0
\in \overline{D(\partial \Psi)}$.

Suppose $X$ is a metric space and $A_n
\subset X$. We define
$$\liminf_{n \to \infty} A_n = \{ x \in X \ : \ \exists x_n \in A_n, \ x_n \to x \} $$
and
$$\limsup_{n \to \infty} A_n = \{ x \in X \ : \ \exists x_{n_k} \in A_{n_k}, \ x_{n_k} \to x \}.$$
In the case $X$ is a normed space, we note by $s$-$\lim$ and
$w$-$\lim$ the above limits associated respectively to the strong
and to the weak topology of $X$.

The {\it epigraph} of a functional $\Psi: H \rightarrow (- \infty, + \infty]$ is defined as
$$ {\rm Epi}(\Psi) := \{(u, \lambda) \in H \times \R \ : \  \lambda \ge \Psi(u)  \}.$$ The following convergence was studied by Mosco in \cite{Mosco}
(see \cite{Attouch}).
 Given a sequence $\Psi_n, \Psi : H
\rightarrow (- \infty, + \infty]$ of convex lower semicontinuous
functionals, we say that $\Psi_n$ converges to $\Psi$ in the sense
of Mosco  in $H$  if
\begin{equation}\label{e1Mos}
w\hbox{-}\limsup_{n \to \infty} {\rm Epi}(\Psi_n) \subset {\rm Epi}(\Psi)
\subset s\hbox{-}\liminf_{n \to \infty}{\rm Epi}(\Psi_n).
\end{equation}
It is easy
to see that \eqref{e1Mos} is equivalent to the two following conditions:
\begin{equation}\label{e2Mos}
\forall \, u \in D(\Psi) \ \ \exists u_n \in D(\Psi_n) \ : \ u_n \to u \ \ \hbox{and} \ \
\Psi(u) \geq \limsup_{n \to \infty} \Psi_n(u_n);
\end{equation}
\begin{equation}\label{e3Mos}
\hbox{for every subsequence} \ n_k, \ \hbox{when} \ u_{k}\rightharpoonup u, \hbox{ it
holds} \ \Psi(u) \leq \liminf_{k} \Psi_{n_k}(u_k).
\end{equation}

As consequence of results in \cite{BP} and \cite{Attouch}   the following
result hodls:

\begin{theorem}\label{convergencia1.Mosco} Let $\Psi_n, \Psi : H
\rightarrow (- \infty, + \infty]$  convex lower semicontinuous
functionals. Then the following statements are equivalent:
\begin{itemize}
\item[(i)] $\Psi_n$ converges to $\Psi$ in the sense
of Mosco  in $H$.

\medskip

\item[(ii)] $(I + \lambda \partial \Psi_n)^{-1} u \to (I + \lambda \partial \Psi)^{-1}
u$, \ \ $\forall \, \lambda > 0, \ u \in H$.
\end{itemize}

Moreover, any of these two conditions $(i)$ or $(ii)$ implies that
\begin{itemize}
\item[(iii)] for every $u_0 \in \overline{D(\partial \Psi)}$ and
$u_{0,n} \in \overline{D(\partial \Psi_n)}$ such that $u_{0,n} \to u_0$, and every $f_n,
f \in L^2 (0,T; H)$   with $f_n \to f$, if $u_n(t)$, $u(t)$ are the strong solutions of the
abstract Cauchy problems
$$
\left\{\begin{array}{ll} u_n^{\prime}(t) + \partial \Psi_n
(u_n(t)) \ni f_n, \quad &\hbox{ a.e.} \ \ t \in ]0,T[,
\\[10pt]
u_{n}(0)=u_{0,n},&
\end{array}\right.  $$
and
$$
\left\{\begin{array}{ll} u^{\prime}(t) + \partial \Psi (u(t)) \ni
f, \quad &\hbox{ a.e.} \ \ t \in ]0,T[,
\\[10pt]
u(0)=u_0,&
\end{array}\right.  $$
respectively, then $$u_n \to u \qquad \mbox{ in } \ C([0,T]: H).$$
\end{itemize}
\end{theorem}

Let us also collect some preliminaries and notations concerning
completely accretive operators that will be used afterwards (see
\cite{BCr2}). Let $(\Sigma, \mathcal{B}, \mu)$ be a $\sigma$-finite
measure space, and $M(\Sigma,\mu)$ the space of $\mu$-a.e. equivalent
classes of measurable functions $u : \Sigma\to \R$.
Let
  \begin{displaymath}
    J_0:= \Big\{ j : \R \rightarrow
    [0,+\infty] : \text{$j$ is convex, lower
      semicontinuous, }j(0) = 0 \Big\}.
  \end{displaymath}
For every $u$, $v\in M(\Sigma,\mu)$, we write
  \begin{displaymath}
  u\ll v \quad \text{if and only if}\quad \int_{\Sigma} j(u)
  \,d\mu \le \int_{\Sigma} j(v) \, d\mu\quad\text{for all $j\in J_{0}$.}
\end{displaymath}

\begin{definition}{\rm
  An operator $A$ on $M(\Sigma,\mu)$ is called {\it completely
    accretive} if for every $\lambda>0$ and
  for every $(u_1,v_1)$, $(u_{2},v_{2}) \in A$ and $\lambda >0$, one
  has that
  \begin{displaymath}
    u_1 - u_2 \ll u_1 - u_2 + \lambda (v_1 - v_2).
  \end{displaymath}
  If $X$ is a linear subspace of $M(\Sigma,\mu)$ and $A$ an operator
  on $X$, then $A$ is {\it $m$-completely accretive on $X$} if $A$ is
  completely accretive and satisfies the {\it range
  condition}
\begin{displaymath}
  \textrm{Ran}(I+\lambda A)=X\qquad\text{for some (or equivalently, for
    all) $\lambda>0$.}
\end{displaymath}
}
\end{definition}

We denote
\begin{displaymath}
  L_0(\Sigma,\mu):= \left\{ u \in M(\Sigma,\mu) \ : \
    \int_{\Sigma} \big[\vert u \vert
   - k\big]^+\, d\mu < \infty \text{ for all $k > 0$} \right\}.
\end{displaymath}
The following results were proved in \cite{BCr2}.

\begin{proposition}
  \label{prop:completely-accretive}
   Let $P_{0}$ denote the set of all functions $q\in C^{\infty}(\R)$
   satisfying $0\le T'\le 1$,
    $q'$  is compactly supported, and $0$ is not contained in the
   support ${\rm supp}(q)$ of $q$. Then,
   an operator $A \subseteq L_{0}(\Sigma,\mu)\times
  L_{0}(\Sigma,\mu)$ is  completely accretive if and only if
   \begin{displaymath}
     \int_{\Sigma}q(u-\hat{u})(v-\hat{v})\, d\mu \ge 0
   \end{displaymath}
   for every $q\in P_{0}$ and every $(u,v)$, $(\hat{u},\hat{v})\in A$.
 \end{proposition}

  The following type of operators were introduced in \cite{C}.
 \begin{definition} Let $(X, \Vert \cdot \Vert)$ a Banach lattice. An operator $A$ operator in $X$ is    {\it $T$-accretive} if
 $$\Vert (u - \hat{u})^+ \Vert \leq \Vert (u - \hat{u} + \lambda (v - \hat{v})^+ \Vert \quad \hbox{for} \ \ (u,v), (\hat{u}, \hat{v}) \in A \ \  \hbox{and} \ \lambda >0.$$
\end{definition}
Obviously, every completely accretive operator
 is a $T$-accretive operator. The mild solutions of the abstract
 Cauchy problems associated with $T$-accretive operators satisfies
 a contraction principle. More precisely, we have the following  result (see \cite{BCP} or \cite[Theorem A.56]{AMRTsp}).

 \begin{theorem}\label{ContPrinTT} Let $X$ be a Banach lattice and $A$
 a $m$-accretive operator in $X$. Then, the following are
 equivalent:
 \begin{itemize}
 \item[(i)] $A$ is $T$-accretive.
  \item[(ii)] If $f, \hat f \in L^1(0, T; X)$, and $u$, $\hat u$
  are mild solutions of $u' + A u \ni f$ and $\hat{u}' + A \hat{u} \ni
  \hat{f}$ on $[0, T]$, then for $0\leq s \leq t \leq T$
  $$\Vert (u(t) - \hat{u}(t))^+ \Vert \leq \Vert (u(s) - \hat{u}(s))^+
  \Vert + \int_s^t [u(\tau) - \hat{u}(\tau), f(\tau) - \hat{f}(\tau)]_+ \, d \tau,$$
  where
  $$[u, v]_+:= \lim_{\lambda \downarrow 0} \frac{\Vert (u +
  \lambda v)^+ \Vert - \Vert u^+ \Vert}{\lambda} \leq \Vert v^+ \Vert.$$
  \end{itemize}
\end{theorem}

 Finally, let us state from~\cite{BEG} the result that we will use for the collapsing results.
\begin{theorem}[\cite{BEG}]\label{teodebenilan} Let, for each $n\in \mathbb{N}$, an m-accretive operator $A_n$ defined a Banach space $X$, with homogeneous degree $p_n$ such that $\lim_np_n=+\infty$. Set $$C:=\{x\in X: \exists (x_n,y_n)\in A_n:x_n\to x, y_n\to 0\}$$
and
$\displaystyle X_0:=\overline{\cup_{\lambda>0}\lambda C}.$ Assume $$\exists \lim_n(I+A_n)^{-1}x=:P(x)\quad\forall x\in X_0.$$
Then, $A:=P^{-1}-I$ is an accretive operator on $X$ with $D(A)=C$ and $X_0\subset \hbox{Ran}(I+\lambda A)$ for each $\lambda>0$.
And if $x_n\in \overline{D(A_n)}$ and $x_n\to x\in L C$ for some $L>1$, then, for $v_n$   the mild solution of
\begin{equation}\label{amn01}
\left\{
\begin{array}{ll}
v_t +A_nv\ni 0  & \hbox{on }\  ]0,+\infty[,
\\[10pt]
v(0)= x_n.
\end{array}
\right.
\end{equation}
we have that
$$v_n\to Qx\quad\hbox{uniformly for $t$ is compact subsets of $]0,+\infty[$},$$
where $Q$ is a contration of $X_0$ onto $C$, and $Qx=v(1)$ where
where $v$ is the mild solution of
\begin{equation}\label{amn02}
\left\{ \begin{array}{ll}
\displaystyle  v_{t} +
Av\ni \frac{v}{t}    & \hbox{on } \left]\frac{1}{L},+\infty\right[,
\\[10pt]
v\left(\frac{1}{L}\right)= \frac{1}{L} x.
\end{array}\right.
\end{equation}
\end{theorem}

 \subsection{Weighted graphs}\label{example.graphs}

  We  work with  locally finite  weighted discrete graphs $$G = (V(G), E(G)),$$ where $V(G)$ is a   set of vertices, that we assume countable,  and  $E(G)$ is a   set of edges connecting  some of the vertices; we write $x\sim y$ if there is and edge connecting the vertices $x$ and $y$ (we  assume that there at most one edge between two vertices). On each edge connecting two vertices $x\sim y$,   it is assigned a positive weight $w_{xy} = w_{yx}$. We also write $w_{xy} = 0$ if $(x,y) \not\in E(G)$. We  assume that there are not loops ($w_{xx}=0$).     For $x \in V(G)$ we define the weighted degree at the vertex $x$ as
$$d_x:= \sum_{y\sim x} w_{xy} = \sum_{y\in V(G)} w_{xy}.$$

When all the weights are $1$, $d_x$ coincides with the degree of the vertex $x$ in a graph, that is,  the number of edges containing $x$.  That the graph is locally finite means that every vertex is only contained in a finite number of edges,  that is, $d_x<+\infty$ for all vertex $x$.

 A finite sequence $\{ x_k \}_{k=0}^n$  of vertices of the graph is called a {\it  path} if $x_k \sim x_{k+1}$ for all $k = 0, 1, ..., n-1$. The {\it length} of a path $\{ x_k \}_{k=0}^n$ is defined as the number $n$ of edges in the path. With this terminology, $G = (V(G), E(G))$ is said to be {\it connected} if, for any two vertices $x, y \in V$, there is a path connecting $x$ and $y$, that is, a path $\{ x_k \}_{k=0}^n$ such that $x_0 = x$ and $x_n = y$.  Finally, if $G = (V(G), E(G))$ is connected, the {\it graph distance} $d_G(x,y)$ between any two distinct vertices $x, y$ is defined as the minimum of the lengths of the paths connecting $x$ and $y$. Note that this metric is independent of the weights.  We always assume that $G$ is connected.

For each $x \in V(G)$  we define the following probability measure, called random walk,
$$m^G_x:=  \frac{1}{d_x}\sum_{y \sim x} w_{xy}\,\delta_y.\\ \\
$$
It is not difficult to see that the measure $\nu_G$ defined as
 $$\nu_G(A):= \sum_{x \in A} d_x,  \quad A \subset V(G),$$
is a reversible measure with respect to this random walk,  that is,
$$  dm_x^G(y)d\nu_G(x)=dm_y^G(x)d\nu_G(y).$$
 Our ambient space is the reversible random walk space  $[V(G),\mathcal{B},m^G,\nu_G]$, where $\mathcal{B}$  is the $\sigma$-algebra of subsets of $V(G)$ (see \cite{MST0} or \cite{MSTBook}). For simplicity, we  write  $V=V(G)$.

We  use the definition of the generalized product measure $\nu \otimes m_x$ (see, for instance, \cite{AFP}), which is defined as the measure on $V\times V$ given by
 $$ \nu_G \otimes m^G_x(U) = \sum_{(x,y) \in U} w_{xy},$$
 on  subsets $U$ of $V\times V$.
 According to the above definitions we have:
$$\int_V f(x) d\nu_G(x)=\sum_{x \in V} f(x)d_x, $$
$$\int_V f(x,y) dm^G_x(y)=\frac{1}{d_x} \sum_{y \in V} f(x,y) w_{xy}, $$
and
$$\int_{V \times V}   f(x,y)   d\nu_G \otimes m^G_x(x,y)=\int_{V \times V}  f(x,y)   dm^G_x(y) d\nu_G(x)=\sum_{x\in V}\sum_{y \in V} f(x,y) w_{xy}.$$
We will use integral or summation notation in the article depending on convenience.

\subsection{Nonlocal gradient and divergence operators}

 Given a function $f: V \rightarrow \R$ we define its {\it nonlocal gradient}
$\nabla f: X \times X \rightarrow \R$ as
$$\nabla f (x,y):= f(y) - f(x) \quad \forall \, x,y \in V.$$
Moreover, given $\z : V \times V \rightarrow \R$, its {\it divergence}
${\rm div}_G \, \z : V \rightarrow \R$ is defined as
 $$ {\rm div}_G \,\z (x):= \frac12 \frac{1}{d_x}\sum_{y\sim x}(\z(x,y) - \z(y,x)) w_{xy}.$$
With the above operators, the  graph Laplacian operator is defined as follows:
$$\Delta_{G} u(x):= {\rm div}_G \,(\nabla u)(x) = \frac{1}{d_x}\sum_{y\sim x}w_{xy}(u(y)-u(x)), \quad u\in L^2(V,\nu_G), \ x\in V.$$
This operator (also called the normalized graph Laplacian) has been studied by many authors (see, for example, \cite{BJ}, \cite{AGrigor}, \cite{DK}, \cite{Elmoatazetal}, \cite{Hafiene}). In the next sections we introduce  $p$-Laplacian operators on graphs.

\section{The first model of sandpile growth}
\subsection{The $p$-Laplacian evolution problem}\label{p-LaplacianE}

 We will assume that $p\geq 3$, which   is not important since our aim is to take limits as $p$ goes to $\infty$ in $p$-Laplacian diffusion problems.

 For $u \in L^{p-1}(V,\nu_G)$, we define the following {\it $p$-Laplacian operator} in $G$:
$$\Delta_p^G u (x):= {\rm div}_G (\vert \nabla u \vert^{p-2} \nabla u)(x),$$
that is, $$\Delta_p^G u (x) = \int_V \vert \nabla u(x,y) \vert^{p-2} \nabla u(x,y) dm^G_x(y)  = \frac{1}{d_x} \sum_{y \in V} \vert \nabla u(x,y) \vert^{p-2} \nabla u(x,y) w_{xy}. $$

By the reversibility of  $\nu_G$ respect to $m^G$, it is easy to prove the following {\it integration by parts formula}.

  \begin{proposition}\label{IntBpart1} For $u \in  L^{p-1}(V,\nu_G)$ with $\Delta_p^G u (x)\in L^q(V,\nu_G)$ ($q\ge 1$)  and $v\in L^{q'}(V,\nu_G)$,
\begin{equation}\label{IPF1}
  \int_V \Delta_p^G u (x) \, v(x) d\nu_G(x) = - \frac12 \int_{V \times V} \vert \nabla u(x,y) \vert^{p-2} \nabla u(x,y) \nabla v(x,y) dm^G_x(y) d\nu_G(x).
\end{equation}
\end{proposition}

Note that if $u\in L^{2}(V,\nu_G)\cap L^{p}(V,\nu_G)$, since $p\ge 3$, then $u \in L^{p-1}(V,\nu_G)$ and $\Delta_p^G u (x)\in L^{p'}(V,\nu_G)$. Then, the above formula is true for any  $u\in   L^{2}(V,\nu_G)\cap L^{p}(V,\nu_G)$ and   $v\in   L^p(V,\nu_G)$.

 Consider the {\it evolution problem} in  $[V(G),\mathcal{B},m^G,\nu_G]$:
\begin{equation}\label{NOnLOp.2}
  P_p^G(u_0,f) \quad\left\{\begin{array}{l} \displaystyle u_t (t,x) =
\frac{1}{d_x}\sum_{y\sim x}  |u(t,y)
 - u(t,x)|^{p-2}(u(t,y)
 - u(t,x)) \, w_{xy} + f(t,x), \\[10pt]
u(0,x)=u_0 (x),
\end{array}\right.
\end{equation}
  with $u_0 \in L^2(G,\nu_G)$ and $f \in L^2(0,\infty; L^2(G,\nu_G)).$
  Let us see  that problem $P_p^G(u_0,f)$   is the gradient flow in $L^2(V,\nu_G)$ associated
to the functional
$$
J^G_p(u) = \left\{\begin{array}{ll}\displaystyle\frac{1}{2p} \int_{V \times V} \vert \nabla u(x,y) \vert^p d(\nu_G \otimes m^G_x)(x,y) \quad &\hbox{if} \ u \in L^2(V, \nu_G) \cap L^p(V, \nu_G), \\[10pt] + \infty  &\hbox{if} \ u \in  L^2(V, \nu_G) \setminus L^p(V, \nu_G), \end{array} \right.
$$
Observe that, for $u \in  L^2(V, \nu_G) \cap L^p(V, \nu_G)$,
$$J^G_p(u) = \frac{1}{2p}  \sum_{(x,y) \in V \times V} |u(y)
 - u(x)|^{p} \, w_{xy}. $$
 To characterize $\partial J^G_p$  we introduce the operator $\mathcal{B}^G_p$ in $   L^2(V, \nu_G)\times L^2(V, \nu_G)$ defined as
$$(u,v) \in \mathcal{B}^G_p \iff u \in L^2(V, \nu_G) \cap L^p(V, \nu_G) \ \hbox{and} \ v = -\Delta_p^G u.$$

\begin{remark}\label{conservmasa}\rm
Observe that, for $(u,v) \in \mathcal{B}^G_p$, if $v\in L^1(V,\nu_G)$ then, by the reversibilidad of $\nu_G$ respect to $m^G_x$, we have $\displaystyle\int_Vvd\nu_G=0$.
\end{remark}

\begin{theorem}\label{CharactTh}  We have that $$\mathcal{B}^G_p = \partial J^G_p,$$ it is $m$-completely accretive in $L^2(V, \nu_G)$ and has dense domain.
\end{theorem}
\begin{proof}
 For every $q\in P_{0}$ and every $(u,v)$, $(\hat{u},\hat{v})\in \mathcal{B}^G_p$, by the integration by parts formula \eqref{IPF1}, we have
 $$\int_{V}q(u-\hat{u})(v-\hat{v})\, d\nu = -\int_{V}q(u-\hat{u}) (\Delta_p^G u  -\Delta_p^G \hat{u}) d \nu_G $$ $$=  \frac12 \int_{V \times V} \vert \nabla u(x,y) \vert^{p-2} \nabla u(x,y) \nabla q(u-\hat{u}) dm^G_x(y) d\nu_G(x) $$
  $$-  \frac12 \int_{V \times V} \vert \nabla \hat{u} (x,y) \vert^{p-2} \nabla \hat{u}(x,y) \nabla q(u-\hat{u}) dm^G_x(y) d\nu_G(x) $$ $$= \frac12 \int_{V \times V} \left(\vert \nabla u (x,y) \vert^{p-2} \nabla u(x,y) -\vert \nabla  \hat{u}(x,y) \vert^{p-2} \nabla \hat{u}(x,y) \right) \nabla q(u-\hat{u}) dm^G_x(y) d\nu_G(x) \geq 0. $$
  Then, by Proposition \ref{prop:completely-accretive}, the operator $\mathcal{B}^G_p$ is completely accretive.

  Let see now that $\mathcal{B}^G_p = \partial J^G_p$. Given $(u,v) \in \mathcal{B}^G_p$ and $w \in  L^2(V, \nu_G) \cap L^p(V, \nu_G)$, by the integration by parts formula \eqref{IPF1}  we have
  $$J^G_p(u) + \int_V v(w-u) d\nu = J^G_p(u) - \int_V \Delta_p^G u(w-u) d\nu = J^G_p(u) $$ $$+ \frac12 \int_{V \times V} \vert \nabla u(x,y) \vert^{p-2} \nabla u(x,y) \nabla w(x,y) dm^G_x(y) d\nu_G(x) - \frac12 \int_{V \times V} \vert \nabla u(x,y) \vert^{p} dm^G_x(y) d\nu_G(x)$$ $$= (1 - p)J^G_p(u) + \frac12 \int_{V \times V} \vert \nabla u(x,y) \vert^{p-2} \nabla u(x,y) \nabla w(x,y) dm^G_x(y) d\nu_G(x).$$
  Now, by Young's inequality
  $$\frac12\int_{V \times V} \vert \nabla u(x,y) \vert^{p-2} \nabla u(x,y) \nabla w(x,y) dm^G_x(y) d\nu_G(x) $$ $$ \leq \frac{1}{2p} \int_{V \times V} \vert \nabla w(x,y) \vert^{p}  dm^G_x(y) d\nu_G(x) + \frac{1}{2p'} \int_{V \times V} \left(\vert \nabla u(x,y) \vert^{p-2} \nabla u(x,y) \right)^{p'} dm^G_x(y) d\nu_G(x)$$ $$= J^G_p(w)+ (p-1) J^G_p(u).$$
  Hence
  $$J^G_p(w) - J^G_p(u) \geq \int_V v(w-u) d\nu.$$
  Consequently $(u, v) \in \partial J^G_p$, and $\mathcal{B}^G_p \subset \partial J^G_p$.

  Conversely, let $(u, v) \in \partial J^G_p$. Then, for every $ w \in  L^2(V, \nu_G) \cap L^p(V, \nu_G)$, we have
  $$J^G_p(u + w) - J^G_p(u) \geq \int_V vw d\nu.$$
  Hence, replacing $w$ by $tw$ for $t >0$, we get
 $$\frac{J^G_p(u + tw) - J^G_p(u)}{t} \geq \int_V vw d\nu_G.$$
Then, taking limit as $t \to 0^+$, we obtain that
$$\frac12 \int_{V \times V} \vert \nabla u(x,y) \vert^{p-2} \nabla u(x,y) \nabla w(x,y) dm^G_x(y) d\nu_G(x)\geq \int_V vw d\nu_G.$$
Now, since this inequality is also true for $-w$, we have
$$\frac12 \int_{V \times V} \vert \nabla u(x,y) \vert^{p-2} \nabla u(x,y) \nabla w(x,y) dm^G_x(y) d\nu_G(x)= \int_V vw d\nu_G.$$
Then, applying again the integration by part formula \eqref{IPF1}, we get
$$- \int_V \Delta_p^G u (x) \, w(x) d\nu_G(x) = \int_V vw d\nu_G \quad \forall \, w \in  L^2(V, \nu_G) \cap L^p(V, \nu_G).$$
Therefore, $v = - \Delta_p^G u$ and consequently, $(u,v) \in \mathcal{B}^G_p$.

Finally, by \cite[Proposition 2.11]{Brezis}, we have
$$ D(\partial J^G_p) \subset  D(J^G_p) =  L^2(V, \nu_G) \cap L^p(V, \nu_G) \subset \overline{D(J^G_p)}^{L^2(V, \nu_G)} = \overline{D(\partial J^G_p)}^{L^2(V, \nu_G)},$$
from which  the density of the domain follows.
\end{proof}

Since $P_p^G(u_0,f)$ coincides with the abstract Cauchy problem
\begin{equation}\label{ACB1}
\left\{ \begin{array}{ll}
u'(t) + \mathcal{B}^G_p(u(t)) \ni f(t) \quad \ t \geq 0, \\[10pt]
u(0) = u_0,
 \end{array} \right.
\end{equation}
by the Brezis-Komura theorem (\cite{Brezis}), having in mind Theorem \ref{CharactTh}, we have the following existence and uniqueness result

\begin{theorem}\label{ExitUniq} For any $u_0 \in L^2(V,\nu_G)$ and $f \in L^2(0, T; L^2(V,\nu_G))$ there exists a unique strong solution $u(t)$ of problem $P_p^G(u_0,f)$, that is,  $u \in C([0,T]:L^2(V,\nu_G))\cap W_{\rm loc}^{1,2}(0,T;L^2(V,\nu_G))$, and, for almost all $t \in ]0,T[$,  $u(t) \in  L^2(V, \nu_G)\cap L^p(V, \nu_G)$ and it satisfies $P_p^G(u_0,f)$.
\end{theorem}

\begin{remark}\rm
 Similar problems to $P_p^G(u_0,f)$ have been studied in~\cite{MSTBook} in a more general framework, but the presentation here differs from the one given there. $\Box$
\end{remark}

\subsection{Limit as $p \to \infty$}\label{sect.covergtop}

With a formal calculation, taking limit  as $p \to \infty$  to the functional $J^G_p$  we
arrive to the functional
\begin{equation}\label{e2MazoNN}
J^G_{\infty}(u) =  \left\{ \begin{array}{ll} 0    &
\hbox{if}
\
\  u \in  L^2(V, \nu_G), \ \Vert \nabla u \Vert_{L^\infty(\nu_G \otimes m^G_x)} \leq 1,
\\[10pt]
+ \infty  & \hbox{in other case}.
\end{array} \right.
\end{equation}
If we define
$$\begin{array}{c}\displaystyle K^G_{\infty}:= \Big\{ u \in L^2(V, \nu_G), \ \Vert \nabla u \Vert_{L^\infty( \nu_G \otimes m^G_x)} \leq 1\Big\} \\[8pt]
\displaystyle = \Big\{ u \in L^2(V, \nu_G)   :   \vert u(y) - u(x) \vert \leq 1 \ \ \hbox{if} \ \ x \sim y  \Big\},
\end{array}$$
   the functional $J^G_\infty$ is given by the indicator function of
 $K^G_\infty$, that is, $J^G_{\infty} = I_{K^G_{\infty}}$.
 Then, the  expected  {\it
 limit
 problem for}~\eqref{ACB1} can be written as
$$
  P_{\infty}^G(u_0,f) \quad\left\{\begin{array}{l} \displaystyle f(t,
\cdot)- u_t
(t,.) \in \partial I_{K^G_{\infty}} (u(t,.)), \ \ \ \hbox{ a.e.} \ t \in ]0,T[, \\[10pt]
u(0,x)=u_0 (x),
\end{array}\right.
$$
 for which  $u\in C([0,T]:L^2(V,\nu_G))\cap W_{\rm loc}^{1,2}(0,T;L^2(V,\nu_G))$ is a strong solution if $u(0) =u_0$ and,  for almost all $t$, $u(t) \in K^G_{\infty}$   and it verifies
\begin{equation}\label{solution1}
0 \geq \int_V (f(t,x) - u_t(t,x))(w(x) - u(t,x)) d\nu_G(x) \quad \hbox{for all} \ w \in K^G_{\infty}.
\end{equation}

 \begin{proposition}\label{CAopI} The operator $\partial I_{K^G_{\infty}}$ is  $m$-completely accretive in $L^2(V, \nu)$.
\end{proposition}
\begin{proof} Since $K^G_{\infty}$ is convex and closed in $L^2(V, \nu)$, we have that  is $\partial I_{K^G_{\infty}}$ $m$-accretive in $L^2(V, \nu)$. By \cite[Lemma 7.1]{BCr2}, to see that $\partial I_{K^G_{\infty}}$ is completely accretive we need to show that
\begin{equation}\label{CondCA}
I_{K^G_{\infty}}(u + q(v - u)) + I_{K^G_{\infty}}(\hat{u} - q(v - u)) \leq I_{K^G_{\infty}}(u) +I_{K^G_{\infty}}(v)
\end{equation} for any $u, v \in L^2(V, \nu_G)$ and any $q \in P_0$.
By \cite[Remark 7.7]{BCr2}, \eqref{CondCA} is equivalent to
\begin{equation}\label{CondCAN}
u, v \in K^G_{\infty} \ \ \hbox{and} \ \  k \geq 0 \quad \Rightarrow  \quad u \vee (v - k), \ u \wedge (v + k) \in K^G_{\infty}.
\end{equation}

 Let $\widetilde{K^G_{\infty}}:= \left\{ u:V\to \mathbb{R} :   \vert u(y) - u(x) \vert \leq 1 \ \ \hbox{if} \ \ x \sim y  \right\}.$  Let us prove that, if $u, v \in  \widetilde{K^G_{\infty}} $ and $k \in \mathbb{R}$, then $ u\vee (v - k)  \in \widetilde{K^G_{\infty}}$. Then, since  $u\in \widetilde{K^G_{\infty}}$ implies $-u\in \widetilde{K^G_{\infty}}$, also $u \wedge (v + k) \in \widetilde{K^G_{\infty}}$. Taking $x\sim y$, we distinguish four possibilities.
 Two of them are trivial, these are  when $u(x)\ge v(x)-k$ and $u(y)\ge v(y)-k$, or when $u(x)< v(x)-k$ and $u(y)< v(y)-k$.
  Let us see the case  $u(x)< v(x)-k$ and $u(y)\ge v(y)-k$. Now, in such case, if $u(y)\le u(x)$ then $v(y)-k\le u(y)\le u(x)< v(x)-k$, and consequently, since $|v(y)-k-(v(x)-k)|=|v(y)-v(x)|\le 1$ we have that
 $$|(u\vee (v - k))(x)-(u\vee (v - k))(y)|=|v(x)-k-u(y)|\le 1;$$  the case $u(y)>u(x)$ follows in a easy similar way.
 The case $u(x)\ge v(x)-k$ and $u(y)< v(y)-k$ is also easy.

 Finally, if $u, v \in  L^2(V,\nu_G)$ and $k\ge 0$ then $u\vee (v - k)    \in L^2(V,\nu_G)$ (similarly, $u \wedge (v + k)\in  L^2(V,\nu_G)$). Indeed, since $k\ge 0$, we have that
 $$(u\vee (v - k))^+\le u^+\chi_{\{u\ge v-k\}}+v^+\chi_{\{u< v-k\}}\in L^2(V,\nu_G)$$
 and
 $$(u\vee (v - k))^-\le u^-\in L^2(V,\nu_G).$$

\end{proof}

  Since $J^G_{\infty}$ is convex and lower semicontinuous in $L^2(V,\nu_G)$,  by the Brezis-Komura Theorem (\cite{Brezis}),
 we have the following existence and uniqueness result.

\begin{theorem}\label{ExitUniqinfty} For any $u_0 \in K^G_{\infty}$ and $f \in L^2(0, T; L^2(V,\nu_G))$ there exists a  unique strong solution $u$  of problem $P_{\infty}^G(u_0,f)$. Moreover, if   $f(t, \cdot)\ge 0$, then $u(t) \geq u_0$ and $u_t \geq 0$ for all $t \geq 0$.
\end{theorem}
\begin{proof} By Proposition \ref{CAopI}, we have $\partial I_{K^G_{\infty}}$ is $T$-accretive in $L^2(V, \nu_G)$. Then, by Theorem~\ref{ContPrinTT}, having in mind that $(u_0)_t + \partial I_{K^G_{\infty}}(u_0) \ni 0$ since $u_0\in K^G_{\infty}$, we have
 $$\Vert (u_0 - u(t))^+ \Vert \leq \Vert (u_0 - u_0)^+
  \Vert + \int_0^t [u_0 - u(\tau),0-  f(\tau)]_+ \, d \tau\le 0,$$
  since $[u_0 - u(\tau),0-  f(\tau)]_+\leq \Vert (0-  f(\tau))^+\Vert\leq 0$. Therefore, we get $u(t) \geq u_0$ for all $t \geq 0$. Consequently, also $u(t + s) \geq u(s)$, and hence $u_t \geq 0$.
\end{proof}

 The limit problem $P_{\infty}^G(u_0,f)$ is just the  model~\eqref{NOnLOpiii} for  sandpile growing in weighted graphs described  in the Introduction.
 Observe that in the formulation of \eqref{solution1}  is given in terms of the  measure $\nu_G$, consequently    problem  $P_{\infty}^G(u_0,f)$ takes into account the weights of the graph $G$ through the weighted degree (see Example \ref{Example2}).  Note also that the result is true without any sing condition for $f$.  When the source $f > 0$ the action on $u$ is to increase following the sandpile  model described previously, but when $f < 0$ the action on $u$ is to decrease following an excavation model with similar constraints on the slope of~$u$, but in fact both situations can interplay.  In Subsection~\ref{Sect.Mass.Transfer} we see that the above problem satisfy a mass conservation principle.

  Let us now see that problem $P_{\infty}^G(u_0,f)$ can be approximated by $p$-Laplacian evolution problems as $p$ goes to infinity.   The   proofs of the next results simplify the ones given in~\cite{AMRTsp} for similar problems in $\mathbb{R}^N$ under nonlocal diffusion.

 \begin{theorem} \label{MoscoConv1}   The functionals $J^G_p$
converge, in the sense of Mosco  in $L^2(V, \nu_G)$, to $J^G_\infty$
  as $p\to \infty$.

 \end{theorem}
 \begin{proof}
First, let us check that
\begin{equation}\label{pepe}
 {\rm Epi}(J^G_{\infty}) \subset s\hbox{-}\liminf_{p
\to \infty}{\rm Epi}(J^G_{p}).
\end{equation}
To this end let $(u,\lambda) \in {\rm Epi}(J^G_{\infty}) $. We can
assume that $u \in K_\infty^G$ and $\lambda \geq 0$.
We define
 $$ u_p:= u \quad \hbox{for all} \ p >2,$$
 and $$\lambda_p =
J^G_{p} (u) + \lambda.$$
Since $u \in K_\infty^G$, we have
\begin{equation}\label{1857}\begin{array}{c}
\displaystyle J^G_p(u_p) =   \frac{1}{2p}  \sum_{(x,y) \in V \times V} |u(y)
 - u(x)|^{p} \, w_{xy} \leq  \frac{1}{2p}  \sum_{(x,y) \in V \times V} |u(y)
 - u(x)|^{2}w_{xy}
 \\ \\
 \displaystyle \le \frac{2}{p} \sum_{x \in V  } |u(x)|^{2}d_x \to 0  \quad \hbox{as $ p\to \infty$,}
 \end{array}
\end{equation}
  and we get \eqref{pepe}. In the last inequality we use  that $(a+b)^2\le 2a^2+2b^2$ for all $a,b\in \mathbb{R}$ and reversibility.

\medskip

Finally, let us prove that
\begin{equation}\label{e4T1.2}
 w\hbox{-}\limsup_{p \to \infty} {\rm
Epi}(J^G_{p}) \subset {\rm Epi}(J^G_{\infty}).
\end{equation}
To this end, let us consider a sequence $(u_{p_j}, \lambda_{p_j})
\in {\rm Epi}(J^G_{p_j})$ ($p_j \to \infty$), that is,
$$
J^G_{p_j} (u_{p_j} ) \leq \lambda_{p_j},
$$
with
$$
u_{p_j} \rightharpoonup u  \  \mbox{ and } \  \lambda_{p_j} \to \lambda.
$$
Therefore we obtain that $\lambda\ge 0$, since
$$
0 \leq J^G_{p_j} (u_{p_j} ) \leq \lambda_{p_j} \to \lambda.
$$
On the other hand, we have that
$$
 \sum_{(x,y) \in V \times V}   \left| u_{p_j} (y)
 - u_{p_j} (x) \right|^{p_j} \, w_{xy}  =  2p_j J^G_{p_j} (u_{p_j})   \leq  C
 p_j .
$$
Then, by the above inequality,  if $q_j = \frac{p_j}{2}+1$, we have
$$
\begin{array}{l}
\displaystyle \left( \sum_{(x,y) \in V \times V}
 \left|
u_{p_j} (y)
 - u_{p_j} (x) \right|^{q_j} w_{xy}  \right)^{\frac{1}{q_j}}
\\[10pt]
\displaystyle
=\left( \sum_{(x,y) \in V \times V}
\left(\sqrt{w_{xy}}\right)\left|
u_{p_j} (y)
 - u_{p_j} (x) \right|^{p_j/2}
\left|
u_{p_j} (y)
 - u_{p_j} (x) \right| \sqrt{w_{xy}}  \right)^{\frac{1}{q_j}}
\\[10pt]
\displaystyle  \leq
\left( \sum_{(x,y) \in V \times V}
 w_{xy} \left|
u_{p_j} (y)
 - u_{p_j} (x) \right|^{p_j}
  \right)^{\frac{1}{2q_j}}
 \left( \sum_{(x,y) \in V \times V}
 \left|
u_{p_j} (y) - u_{p_j} (x) \right|^{2}w_{xy}
  \right)^{\frac{1}{2q_j}}
 \\[16pt]
 \displaystyle \leq
 (C
 p_j)^{\frac{1}{p_j+2}}
 \left( \sum_{(x,y) \in V \times V}
 \left|
u_{p_j} (y) - u_{p_j} (x) \right|^{2}w_{xy}
  \right)^{\frac{1}{p_j+2}}
\\[16pt]
 \displaystyle
\le (C
 p_j)^{\frac{1}{p_j+2}}
 \left(2\sum_{x \in V  }  |u_j(x)|^2 d_x\right)^{\frac{1}{p_j+2}} ,
 \end{array}
$$
which is bounded since $u_{p_j} \rightharpoonup u$.
Hence, letting  $j
\to \infty$ (then $ q_j\to+\infty$) we obtain:
\begin{equation}\label{1903}
 | u(x) - u(y) | \leq 1\qquad\hbox{for } x\sim y.
\end{equation}
And we conclude that
$$
u \in K^G_\infty.
$$
Then, \eqref{e4T1.2} holds,
which ends the proof.
 \end{proof}

As consequence of the above theorem and Theorem
\ref{convergencia1.Mosco} we have the following result,

\begin{theorem} \label{convergencia.p.intro}
  Let   $T>0$, $f \in L^2(0,T; L^2(V,\nu_G))$,
 $u_0\in K^G_{\infty}$,
and $u_{p}$ be the
unique solution of $P_p^{G}(u_0,f)$. Then,   if
 $u_{\infty}$ is the unique solution to
$P^G_{\infty}(u_0,f)$,
$$
\displaystyle \lim_{p\to \infty}\sup_{t\in [0,T]}\Vert
u_{p}(t,\cdot)-u_{\infty}(t,\cdot)\Vert_{L^2(V,\nu_G)}=0.
$$
\end{theorem}

\subsection{Collapse of the initial condition}\label{Sect.collapsing}

As we mentioned in the Introduction, in \cite{EFG}, Evans,
Feldman and Gariepy
 study the behavior of
 the solution $v_p$ of the initial value problem
 \begin{equation}\label{p_Lap1.introno}
\left\{\begin{array}{ll} v_{p,t}- \Delta_p v_p = 0, \quad & t \in
]0,T[ ,
\\[10pt]
v_p(0,x)=u_0(x), \quad & x \in \R^N,
\end{array}\right.
\end{equation}
as $p \to \infty$,
 when the initial condition $u_0$  is a Lipschitz function
with compact support  satisfying
$$\hbox{ess sup}_{\R^N} \vert \nabla u_0 \vert = L >  1.$$
They prove that for each time $t > 0$
$$v_p(t, \cdot) \to v_{\infty}(\cdot), \ \ \ \ \hbox{uniformly as}
\ \ p \to + \infty,$$ where $v_{\infty}$ is independent of time
and satisfies
$$\hbox{ess sup}_{\R^N} \vert \nabla v_{\infty} \vert \leq 1.$$
Moreover, $v_{\infty}(x) = v(1,x)$,  $v$ solving the nonautonomous
evolution equation
$$
\left\{ \begin{array}{ll}
\displaystyle \frac{v}{t} - v_{t} \in
\partial I_{K_{0}}( v), \quad & t \in ]\tau,\infty[
\\[10pt]
v(\tau,x)= \tau u_0(x),
\end{array}\right.
$$
where $\tau = L^{-1}$  and
$$K_0:= \left\{ u \in L^2(\R^N)\cap W^{1,\infty}(\R^N) \ : \ \vert \nabla
u \vert \leq 1  \right\}.$$
They interpreted the above result as a crude model for
the collapse of a sandpile from an initially unstable
configuration. The proof of this result is based in a scaling
argument, which was extended by  B\'{e}nilan,  Evans  and  Gariepy  in
\cite{BEG}  to cover general nonlinear evolution equations
governed by homogeneous accretive operators  (see Theorem~\ref{teodebenilan} in Section~\ref{sect.prelim}). Here, by using such result, we get a collapsing  sandpile  model on weighted graphs.

Let $p\ge 3$. We look for the limit as $p\to \infty$ of the solutions $u_p$ to the
 problem  $P_p^G(u_0,0)$ when the  initial datum $u_0$
satisfies
\begin{equation} \label{L1006}
1< L := \Vert \nabla u_0\Vert_{L^\infty( \nu_G \otimes m^G_x)}<\infty.
\end{equation}
 The solution $u_p$ to the
 problem  $P_p^G(u_0,0)$ coincides with the strong solution
of the abstract Cauchy problem
\begin{equation}\label{NOnLOpoioi}
\left\{\begin{array}{ll}    u_t +
\partial J^G_{p}u \ni 0  \quad & \hbox{on }\ ]0,T[, \\[10pt]
u(0,x)=u_0 (x), \quad & x \in V.
\end{array}\right.
\end{equation}
 In Theorem~\ref{teodebenilan}, this problem corresponds to Problem~\eqref{amn01}. Let us see that we have the ingredients to apply such result:
\\[5pt]
1. The
operators $\partial J^G_{p}$  are
$m$-accretive operators in $L^2(V,\nu_G)$ (Theorem \ref{CharactTh}) and   positively
homogeneous of degree $p-1$.
\\[5pt]  2. Set
$$C:= \left\{ u \in L^2(V,\nu_G) \  : \ \exists (u_p, v_p) \in \partial
J^G_{p} \ \mbox{with} \ u_p \to u, \ v_p  \to 0 \ \hbox{as} \ p
\to \infty \right\}.$$
 Let us characterize this set. We have that
\begin{equation}\label{CequalK} C = K^G_{\infty} = \left\{ u \in L^2(V,\nu)  :
 \vert u(x) - u(y) \vert \leq 1, \  \nu_G \otimes m^G_x\hbox{-a.e.} \ (x,y) \in V \times V \right\}.
 \end{equation}
 In fact, if $u \in  K^G_{\infty}$, we have $u = (I + \partial I_{K^G_{\infty}})^{-1}u$. Then, by Theorem \ref{MoscoConv1} and Theorem~\ref{convergencia1.Mosco}, we have
 $$u_p:= (I + \partial J^G_p)^{-1} u \to u \quad \hbox{in} \quad L^2(G, \nu_G), \quad \hbox{as} \ p
\to \infty.$$
 Moreover, $$v_p:=  u- u_p \in \partial J^G_p(u_p), \quad v_p \to 0 \quad \hbox{in} \quad L^2(G, \nu_G) \quad \hbox{as} \ p
\to \infty.$$
 Therefore, $K^G_{\infty} \subset C$.
 Suppose now that $u \in C$, and let $$(u_p, v_p) \in \partial
J^G_{p} \quad  \mbox{with} \ u_p \to u, \ v_p  \to 0 \quad \hbox{as} \ p
\to \infty.$$ Again, we have
 $$(I + \partial J^G_p)^{-1} u \to (I + \partial I_{K^G_{\infty}})^{-1}(u) \quad \hbox{in} \quad L^2(G, \nu_G) \quad \hbox{as} \ p
\to \infty.$$
On the other  hand, since $(I + \partial J^G_p)^{-1}$ are contraction in $L^2(G, \nu_G)$, we have
$$u_p = (I + \partial J^G_p)^{-1}(u_p +v_p) \to (I + \partial J^G_p)^{-1}(u) \quad \hbox{in} \quad L^2(G, \nu_G) \quad \hbox{as} \ p
\to \infty.$$
Therefore, $u = (I + \partial I_{K^G_{\infty}})^{-1}(u)$ and consequently $u \in K^G_{\infty}$.
\\[5pt]
3. Since for $u\in L^2(G, \nu_G)$ and $\lambda>0$  we have that $T_\lambda u\in \lambda C$, then
 \begin{equation}\label{closure1}  \overline{\bigcup_{\lambda > 0} \lambda C}^{L^2(V,\nu_G)} =
 L^2(V,\nu_G).\end{equation}
\\[5pt]
4. By the Mosco convergence (Theorem~\ref{MoscoConv1}), for $f \in L^2(V,\nu_G)= \overline{\bigcup_{\lambda > 0} \lambda C}^{L^2(V,\nu_G)}$ and
 $  v_p:= (I + \partial J^G_{p})^{-1} f$, there exists a sequence
 $p_j \to + \infty$  such that $v_{p_j} \to (I
+ \partial I_{K^G_{\infty}})^{-1}f$ in $L^2(V,\nu)$ as
 $j \to \infty$. Therefore,
$$\exists\lim_{p \to \infty} (I + \partial J^G_{p})^{-1} f=(I
+   \partial I_{K^G_{\infty}})^{-1}f.$$

 Hence $(I
+ \partial I_{K^G_{\infty}})^{-1}$ is the operator $P$ in Theorem~\ref{teodebenilan}, and   $A=\partial I_{K^G_{\infty}}$, that we already know is $m$-accretive and has $C$ as domain.
\\[5pt]
 5. Finally, for $u_0 \in L^2(G,\nu_G)$  satisfying~\eqref{L1006}, we have $u_0\in LC$. Then,  by Theorem~\ref{teodebenilan}  and the regularity of the solutions (\cite[Theorem 3.6]{Brezis}), we obtain the following result.

\begin{theorem} \label{teo.collapsing.intro}
Let $u_p$ be the solution to $P_{p}^G(u_0,0)$ with initial condition
$u_0 \in L^2 (V, \nu_G)$ such that
$$
1< L := \Vert \nabla u_0 \Vert_{L^\infty( \nu_G \otimes m^G_x)}<\infty.
$$
Then, there exists the limit
$$
\lim_{p\to \infty } u_p (t,x) = u_\infty (x) \qquad \mbox{ in }
L^2(V,\nu_G),
$$
which is a function independent of $t$ such that $u_\infty \in K^G_{\infty}$. Moreover,
$u_\infty(x) = v(1,x)$, where $v$ is the unique   strong solution of
the evolution equation
\begin{equation}\label{1312}
\left\{ \begin{array}{ll}
\displaystyle \frac{v}{t} - v_{t} \in \partial
I_{K^G_{\infty}}(v), \quad & t \in ]\tau,\infty[,
\\[10pt]
v(\tau,x)= \tau u_0(x),
\end{array}\right.
\end{equation}
with $\tau = L^{-1}$.
\end{theorem}

Following the same arguments of~\cite{EFG}, Problem~\eqref{1312} can be seen as a {\it weak} sandpile model to obtain the collapsing  of a sandpile $u_0$ that violates the slope condition $\Vert \nabla u_0 \Vert_{L^\infty( \nu_G \otimes m^G_x)}\le 1$.

\begin{proposition}\label{vpositiva}  Under the conditions of Theorem~\ref{teo.collapsing.intro} , if $0\le u_0\in L^2(V,\nu_G)$ then $v(t)\ge 0$. Moreover, it is nondecreasing in time.
\end{proposition}

\begin{proof}
  Applying Theorem \ref{ContPrinTT}, for $t > \tau$, we have
$$\Vert (- v(t))^+ \Vert \leq \Vert (- \tau u_0)^+
  \Vert + \int_{\tau}^t \frac{1}{s} \Vert (- v(s) )^+ \Vert \, ds = \int_{\tau}^t \frac{1}{s} \Vert (- v(s) )^+ \Vert \, ds.$$
Hence, by Gr\"{o}nwall's Lemma, $$\Vert (- v(t))^+ \Vert\le 0,$$
from where $v(t)\ge 0$. Finally, from Theorem~\ref{ExitUniqinfty}, $v_t \geq 0$ for all $t \geq \tau$.
\end{proof}

\begin{proposition}\label{remL1}\rm
 Under the assumptions of Theorem~\ref{teo.collapsing.intro}  we have that, if moreover $u_0\in L^1(V,\nu_G)$, then $u_\infty \in L^1(V,\nu_G)$.
 \end{proposition}

\begin{proof}
Since the operator $\partial J^G_{p}$ is completely accretive, we have that the solution $u_p$  to $P_{p}^G(u_0,0)$ satisfy   $u_p(t) \ll u_0$  for all $t \geq 0$. Then, by \cite[Propposition 2.11]{BCr2}, we have $u_\infty \in L^1(V,\nu_G)$.
\end{proof}

\subsection{Mass transport interpretation}\label{Sect.Mass.Transfer}

In  \cite{EFG} or \cite{Evans}, a mass transfer interpretation of the
limit problem $P_\infty (u_0,f)$ is described. Our aim in this section is to give also an
  explanation of the limit problem $P^G_\infty (u_0,f)$ by using mass transport theory.

Consider the metric space $(V,d_G)$, where $d_G$ was defined in Subsection~\ref{example.graphs}.  Let   $f_0,  f_1  $ be two nonnegative $L^1$ functions in $V$ with equal mass, i.e.,
$$\int_V f_0(x) d\nu_G(x) =  \int_V f_1(x) d\nu_G(x).$$
  Let $\mathcal{A}(f_0,f_1)$ be the set of transport maps pushing $f_0$ to $f_1$, that is, the set of Borel maps
$T: V \to V$ such that $\int_V f_0\circ Td\nu_G = \int_Vf_1 \nu_G$. The {\it Monge problem} consists in finding a map $T^* \in \mathcal{A}(f_0,f_1)$
which minimizes the cost functional
$$  \int_V  f_0(x)d_G(x,T(x)) d \nu_G(x)$$
in the set $\mathcal{A}(f_0,f_1)$. The map $T^*$ is called an {\it optimal transport map} pushing $f_0$ to $f_1$.

A relaxation of the Monge problem, proposed by Kantorovich \cite{K} is
the {\it Monge-Kantorovich transport problem}
 associated to the distance $d_G$ given by:
minimize
\begin{equation}\label{minimiz.M-K}
    \int_{V \times V} d_G(x,y) d\gamma(x,y)
\end{equation}
among  all transport plans between $f_0$ and $f_1$, that is, Radon measures in $V \times V$, such that $\pi_1 \sharp \gamma =f_0 d\nu_G$ and $\pi_2 \sharp \gamma =f_1 d\nu_G$, that we denote by $\Pi(f_0,f_1)$.
  It is well-known (see \cite{A}) that
\begin{equation}\label{inequ1}
\inf_{\gamma \in \Pi(f_0,f_1)}    \int_{V \times V} d_G(x,y) d\gamma(x,y)  \leq \inf_{ T \in \mathcal{A}(f_0,f_1)}   \int_V  f_0(x)d_G(x,T(x)) d \nu_G(x) . \end{equation}
On the other hand, since $d_G$ is a lower semicontinuous cost function, it is well known the
existence of an {\it optimal transport plan}, that is, a $\gamma^*  \in \Pi(f_0,f_1)$ such that
$$\int_{V \times V} d_G(x,y) d\gamma^*(x,y) = \inf_{\gamma \in \Pi(f_0,f_1)}    \int_{V \times V} d_G(x,y) d\gamma(x,y).$$

 The dual formulation of the Monge-Kantorovich transport problem  is given by
$$
\max_{u \in K_{d_G}} \int_{V} u(x) (f_1 (x)- f_0(x)) d\nu_G(x)
$$
where
$$
K_{d_G}:= \left\{ u :V \rightarrow \R \ : \
 \vert u(x) - u(y) \vert \leq d_G(x,y) \ \ \forall \,x,y \in V \right\}.$$
The Kantorovich duality Theorem (see \cite{Villani}) establishes that
\begin{equation}\label{transesta}
 \inf_{\gamma \in \Pi(f_0,f_1)}    \int_{V \times V} d_G(x,y) d\gamma(x,y)= \max_{u \in K_{d_G}} \int_{V} u(x) (f_1 (x)- f_0(x)) d\nu_G(x).
\end{equation}
 The function $u$ that maximize the above problem is called a {\it Kantorovich potential} of the transport problem \eqref{minimiz.M-K}.

  Working as in the proof of \cite[Lemma 9]{CAFAetal}   (see also~\cite[Lemma 2.9]{IMRTpasos}), we have the following
 {\it Dual Criteria for Optimality}.
 \begin{lemma}\label{CritOpt}
 \noindent $(1)$ If $u^* \in K_{d_G}$ and $T^* \in \mathcal{A}(f_0,f_1)$ satisfy
 \begin{equation}\label{E1CritOpt}
 u^*(x)- u^*(T^*(x)) = d_G(x,T^*(x)) \quad \hbox{for all} \ \ x \in \hbox{supp}(f_0),
 \end{equation}
 then
 \begin{itemize}
 \item[(i)] $u^*$ is a kantorovich potential for the metric $d_G$.

  \item[(ii)] $T^*$ is an optimal map for the Monge problem associated to the metric $d_G$.
  \item[(iii)] $$\inf_{ T \in \mathcal{A}(f_0,f_1)}   \int_V  f_0(x) d_G(x,T(x)) d \nu_G(x)  = \sup_{u \in K_{d_G}} \int_{V} u(x) (f_1 (x)- f_0(x)) d\nu_G(x).$$
  \end{itemize}
  \noindent $(2)$ Under (iii), every optimal map $\hat{T}$ for the Monge problem associated to the metric $d_G$ and
Kantorovich potential $\hat{u}$ for the metric  $d_G$ satisfy \eqref{E1CritOpt}.

\end{lemma}

Let $u(t, \cdot)$ be the unique solution to
$P^G_{\infty}(u_0,f)$.  In the case $\nu_G(V) < \infty$, we have the following {\it conservation of the mass principle}:
\begin{equation}\label{ConMass1}\int_Vu_{\infty,t}(t,x)d\nu_G(x)= \int_Vf(t,x)d\nu_G(x) \quad  \hbox{for all $t \geq 0$}.\end{equation}
In fact, since $u(t, \cdot)$ is a unique solution to
$P^G_{\infty}(u_0,f)$, we have $u(t) \in K^G_\infty$ and
$$0 \geq \int_V (f(t,x) - u_t(t,x))(w(x) - u(t,x)) d\nu_G(x) \quad \hbox{for all} \ w \in K^G_{\infty}.$$
Now, since $u(t) \in K^G_\infty$ and $\nu_G(V) < \infty$, we have $w(x):= u(t) \pm\1_V \in K^G_\infty$, thus \eqref{ConMass1} holds.
Now,  when $\nu_G(V)$ is not finite, the conservation of mass also holds true.

\begin{theorem}\label{consmass}  Let $ u_0 \in K^G_\infty\cap L^1(V,\nu_G)$ and   $f \in L^2(0,T; L^2(V,\nu_G)\cap L^1(V,\nu_G))$. Then, if $u_{\infty}$ is the unique solution to
$P^G_{\infty}(u_0,f)$, we have
$$\int_Vu_{\infty,t}(t,x)d\nu_G(x)= \int_Vf(t,x)d\nu_G(x) \quad  \hbox{for all $t \geq 0$}.$$
\end{theorem}

 Before giving its proof, let us give the relation of $u_\infty$ with mass transport. Let $0\le  u_0 \in K^G_\infty\cap L^1(V,\nu_G)$ and   $0\le f \in L^2(0,T; L^2(V,\nu_G)\cap L^1(V,\nu_G))$,
by Theorem~\ref{ExitUniqinfty} and the above Theorem~\ref{consmass},  if $u_\infty (t, \cdot)$ is the solution  of  problem
$P^G_\infty (u_0,f)$, then  $u_{\infty,t}\ge 0$ and   the masses of $u_t$ and $f$ are equal.  Now, since $K_{d_G} = K^G_{\infty}$, we have
$$0\geq \int_{V} (f (t,x) - u_{\infty,t}(t,x)) (v(x) - u_\infty(t,x))\, d\nu_G(x) \quad \hbox{for every} \ v \in K_{d_G},$$
hence
$$\int_{V} u_\infty(t,x)(f (t,x) - u_{\infty,t} (t,x)) \, d\nu_G(x) = \max_{v \in K_{d_G}} \int_{V} v(x) (f (t,x) - u_{\infty,t} (t,x)) \, d\nu_G(x).$$
Therefore, we have that $u_\infty(t,\cdot)$ is a Kantorovich potential for the transport problem, respect the distance $d_G$, between  the source $f(t,\cdot)$
 and the time derivative of the solution, $u_{\infty,t}
(t,x)$. Consequently, we conclude that the mass of sand added by the
source $f(t, \cdot)$ is transported (via $u_\infty(t, \cdot)$ as the
transport potential) to $u_{\infty,t}(t,\cdot)$ at each time~$t$.
 Consequently, we have the following result.
\begin{theorem}\label{TrabspSol}   Let  $ u_0 \in K^G_\infty\cap L^1(V,\nu_G)$ and   $0\le f \in L^2(0,T; L^2(V,\nu_G)\cap L^1(V,\nu_G))$. Let     $u \in C([0,T]: L^2(V,\nu_G)\cap L^1(V,\nu_G))\cap W_{\rm loc}^{1,2}(0,T;L^2(V,\nu_G))$,  $u(t) \in K^G_{\infty}$   for   $t \in ]0,T[$, such that $u(0,\cdot) = u_0$.   Then, $u$ is a strong solution of problem $P_{\infty}^G(u_0,f)$ if and only if $u(t, \cdot)$  is a Kantorovich potential for the transport problem, respect the distance $d_G$, between  the source $f(t,\cdot)$
 and the time derivative $u_{t}(t, \cdot)$.
\end{theorem}

Let us now proof that mass is preserved.  We will use the following notation. For a $A\subset V$, its nonlocal boundary is
$$\partial_{m^G}A=\{y\in V\setminus A :y\sim x \hbox{ for some }x\in A\}.$$

\begin{proof}[Proof of Theorem~\ref{consmass}]  {\it Step 1}. Suppose first that $u_0$ and $f(t)$ have finite support and are bounded. Then,
we can assume without loss of generality that $u_0$ and $f$ are like follows  (by adding null constants $\alpha_i$ or null functions $f_i$):
$$u_0= \sum_{i=1}^n \alpha_i \1_{\{x_i\}} \quad \hbox{and} \quad f(t)=  \sum_{i=1}^n  f_i(t)\1_{\{x_i\}},$$
 with $\alpha_i \in \mathbb{R}$ and $-M\le f_i(t)\le M$ for all $ i=1, \dots, n,$ for some  $M>0$.

 Let us see first that we can estimate the support $u_\infty(t,.)$ and of $u_{\infty,t}(t,.)$. Set $\{y_j\}_{j=1}^k=\partial_{m^G}\{x_i\}_{i=1}^n$. Take
$$v(t,x)= \sum_{i=1}^n (\alpha_i+Mt) \1_{\{x_i\}}(x)+\sum_{j=1}^k Mt \1_{\{y_j\}}(x),$$ then we have that $v(t,.)\in K^G_{\infty}$ for $0\le t\le \frac{1}{M}$, and
$$v_t(t,x)=\sum_{i=1}^n M \1_{\{x_i\}}(x)+\sum_{j=1}^k M \1_{\{y_j\}}(x)=:\tilde f(t,x).$$
Then, from $0\le t\le \frac{1}{M}$, $v$ is solution to $P^G_{\infty}(u_0,\tilde f)$. Now, since $f\le \tilde f,$ we have that
$$u_\infty(t,x)\le v(t,x)\quad\hbox{for} \ \  0\le t\le\frac{1}{M}.$$
Indeed, applying Theorem \ref{ContPrinTT}, we have
$$\Vert (u_\infty(t) - v(t))^+ \Vert \leq \Vert (u_0 - u_0)^+
  \Vert + \int_0^t \Vert (f(\tau) - \tilde f(\tau))^+\Vert \, d \tau \leq 0.$$
 Similarly, if
$$w(t,x)= \sum_{i=1}^n (\alpha_i-Mt) \1_{\{x_i\}}(x)-\sum_{j=1}^k Mt \1_{\{y_j\}}(x),$$  then $w(t,.)\in K^G_{\infty}$ for $0\le t\le \frac{1}{M}$, and
$$w_t(t,x)=-\sum_{i=1}^n M \1_{\{x_i\}}(x)-\sum_{j=1}^k M \1_{\{y_j\}}(x)=:\hat f(t,x).$$
Then, from $0\le t\le \frac{1}{M}$, $w$ is solution to $P^G_{\infty}(u_0,\hat f)$. Now,  $f\ge \hat f,$ and we get
$$u_\infty(t,x)\ge w(t,x)\quad\hbox{for} \ \  0\le t\le\frac{1}{M}.$$
  Therefore, $$ w(t)\le u_\infty(t) \leq v(t)  \quad\hbox{for} \ \ 0\le t\le\frac{1}{M}.$$
Hence the supports of $u_\infty(t,.)$ and $u_{\infty,t}(t,.)$ are contained in $\{x_i\}_{i=1}^n\cup \partial_{m^G}\{x_i\}_{i=1}^n.$

Since $u(t, \cdot)$ is a unique solution to
$P^G_{\infty}(u_0,f)$, we have $u(t) \in K^G_\infty$ and
$$0 \geq \int_V (f(t,x) - u_t(t,x))(w(x) - u(t,x)) d\nu_G(x) \quad \hbox{for all} \ w \in K^G_{\infty}.$$

 Let $A_t$ be the support of $u(t,.)$.  Then   $w(x):= u(t) \pm \1_{A_t \cup \partial_{m^G}(A_t)} \in K^G_\infty$  for $0 \leq t \le \frac{1}{M}$ and consequently, we get
$$\int_Vu_{\infty,t}(t,x)d\nu_G(x)= \int_Vf(t,x)d\nu_G(x) \quad  \hbox{for  $0 \leq t \le \frac{1}{M}$ }.$$
 We can repeat the above argument to cover any time interval.

    {\it Step 2.} Consider now the general case. It is easy to see that there exists  $0 \leq f_n(t)\in  L^\infty(0,T; L^\infty(V, \nu_G))$ with finite support such that $f_n \to f$ in $L^2(0,T; L^1(V, \nu_G))$.  Let us see that  there exist $ u_{0n} \in K^G_{\infty} \cap L^1(V, \nu_G)$ with ${\rm supp}(u_{0n})$ finite such that $u_{0n} \to u_0$ in $L^1(V, \nu_G)$. We prove this in two steps.
Firstly we approximate $u_0$ by a sequence in $w_n\in K^G_{\infty} \cap L^1(V, \nu_G)$ such that, for each $n$,  $\hbox{supp}(w_n{\!}^+)$ has finite support. Indeed,
     let $x_1\in V$ and define
   $$A_1=\{x_1\},$$
   $$A_2=A_1\cup \partial_{m^G}A_1,$$
   $$A_n=A_{n-1}\cup \partial_{m^G}A_{n-1}\quad \hbox{for any } n\in \mathbb{N}.$$
    Observe that, by connectedness, $V=\cup_{n=1}^\infty A_n$. Define now, for each $n\in \mathbb{N}$,
 $$ m_n=\max_{x\in A_n}u_0(x)^+, \quad k_n=[m_n],$$
 $$v_n(x)=k_n\chi_{A_n}+(k_n-1)\chi_{A_{n+1}\setminus A_{n}}+ \cdots +1\chi_{A_{n+k_n-1}\setminus A_{n+k_n-2}}+0\chi_{A_{n+k_n}\setminus A{n+k_n-1}},$$
 and
 $$w_n=v_n \wedge u_0.$$
 Then, it is easy to see that  $\{w_n\}_n$ is a nondecreasing sequence,  bounded from above by $u_0{\!}^+$, and converging punctually to $u_0$.
    Hence, by the dominated convergence theorem,  $w_{n} \to u_0$ in $L^1(V, \nu_G)$.
Moreover, we have that
 each $w_n\in K^G_\infty\cap L^1(V,\nu_G)$   (see the proof of Proposition~\ref{CAopI}) and the support of $w_n{\!}^+$ is finite.
 In the second step,  for any $w\in K^G_{\infty} \cap L^1(V, \nu_G)$  whose nonnegative part has finite support, we can find, working in a similar way,  $ \widetilde w_n\in  K^G_\infty\cap L^1(V,\nu_G)$ with $\tilde w_n{\!}^-$ having finite support, and hence with $w_n$ having finite support, such that $\widetilde w_{n} \to w$ in $L^1(V, \nu_G)$.
  Consequently, we can find $ u_{0n} \in K^G_{\infty} \cap L^1(V, \nu_G)$ with ${\rm supp}(u_{0n})$ finite such that $u_{0n} \to u_0$ in $L^1(V, \nu_G)$.

  Let $u_n(t)$ be the solution of problem $P^G_{\infty}(u_{0n},f_n)$. By {\it Step 1}, we have
 \begin{equation}\label{Consmassn}
 \int_V u_n(t) d \nu_G = \int_V u_{0n} d \nu_G + \int_0^t \int_V f_n(s) ds \,  d \nu_G.
 \end{equation}
 Now, by the complete accretivity of $\partial I_{K^G_{\infty}}$, we have $u_n(t) \to u_\infty(t)$ in $L^1(V, \nu_G)$ as $n \to \infty$. Then, taking limits in \eqref{Consmassn}, we get

 $$\int_V u_\infty(t) d \nu_G = \int_V u_{0} d \nu_G + \int_0^t \int_V f(s) ds \,  d \nu_G,$$
 and consequently
 $$\int_Vu_{\infty,t}(t,x)d\nu_G(x)= \int_Vf(t,x)d\nu_G(x) \quad  \hbox{for all} \ \  t \geq 0.$$
\end{proof}

\begin{remark}\label{remsoporte}\rm
We want to remark  that the mass conservation principle can be  used independently on subgraphs,  that is,
while the sandpile growth of each subgraph is independent of each other.  Example~\ref{ej23} illustrates this observation.
\end{remark}

\subsection{Explicit solutions}\label{sect.explicit}

 In this section we show some explicit simple examples of solutions to the  sandpile model
\begin{equation}\label{NOnLOpExpl}
P_{\infty}^G(u_0,f)\quad\left\{\begin{array}{l} \displaystyle f(t,
\cdot)- u_t
(t,.) \in \partial I_{K^G_{\infty}} (u(t,.)), \ \ \ \hbox{ a.e.} \ t \in ]0,T[, \\[10pt]
u(0,x)=u_0 (x)
\end{array}\right.
\end{equation}
 that illustrate  the dynamic involved in this model.

In order to verify that a function $u(t,x)$ is a solution to
$P_{\infty}^G(u_0,f)$ we need to check that
\begin{equation}\label{def.solution}
I_{K^G_{\infty}} (v) \ge I_{K^G_{\infty}}  (u) + \langle
f-u_t , \ v-u \rangle, \qquad \mbox{ for all } v \in L^2 (V, \nu).
\end{equation}
To this end we can assume that $v \in K^G_{\infty} $ (otherwise $K^G_{\infty}
(v) = + \infty$ and then \eqref{def.solution} becomes trivial). Therefore, we need to
show
$$
    u(t,\cdot) \in K^G_{\infty}
$$
and
\begin{equation}\label{def.solution.3}
    0\geq \int_{V} (f (t,x) - u_t (t,x)) (v(x) - u(t,x))\, d\nu_G(x) \qquad \hbox{for every $v \in K^G_{\infty}$}.
\end{equation}
  By  the mass conservation principle (Theorem \ref{consmass}),  we have that
$$\int_Vf (t,x)\, d\nu_G(x) =\int_V u_t (t,x))\, d\nu_G(x),$$
that is,
$$\sum_{x \in V} f(t,x)d_x=\sum_{x \in V} u_t(t,x)d_x.$$
This principle, joint the fact that solutions must belong to $K^G_{\infty}$, gives us a way to find such solutions.

\begin{example} {\rm Let us consider, the weighted graph $\Z$ with weights $$w_{xy} =\left\{ \begin{array}{ll} 1 & \hbox{if } \vert x - y \vert =1,\\[8pt]
0&\hbox{otherwise}.\end{array}\right.$$  We take
as source the function
$$
    f(t,x) = \alpha \1_{\{0 \}}(x),
    \quad   \alpha> 0,
$$
and as initial datum
$$
    u_0 (x) = 0.
$$
Let us find the sandpile growing solution to
$P_{\infty}^G(u_0,f)$ by looking at its evolution between
some critical times.

$\bullet$ First, for small times, the solution to
$P_{\infty}^G(u_0,f)$ is clearly given by
$$
u(t,x) =  \alpha t\1_{\{0 \}}(x), \quad \hbox{for} \ 0 = t_0 \leq t < t_1 = \frac{1}{\alpha}.
$$
Remark that $t_1 = \frac{1}{\alpha}$ is the first time when
$u(t,x) = 1$ in $0$. It is immediate that $u(t,\cdot)
\in K^G_{\infty}$ and $u_t(t,x) = f(t,x)$, then
\eqref{def.solution.3} holds.

 Observe that for a source $f(t,x)=\tilde f(t)\1_{\{ 0\}}$ then $u(t,x) =  \int_0^t \tilde f(\tau)d\tau\1_{\{0 \}}(x)$, for $0 \leq t < t_1$ with $\int_0^{t_1}\tilde f(\tau)d\tau=1$. We  only give the examples in the simple situation of constant sources.

$\bullet$ For times greater than $t_1$ the support of the solution is
greater than the support of~$f$. Indeed the solution can not be
larger than $1$ in $0$
without being larger than zero in the adjacent vertices $x= \pm1$;
more concretely, it must belong to $K^G_{\infty}$. So, let us  see that  the sandepile growing solution has the form:
$$u(t,x) = \displaystyle  \widetilde\alpha(t)\1_{\{-1 \}}(x)+ \left(1 + \widetilde\alpha(t) \right) \1_{\{0 \}}(x) +  \widetilde\alpha(t)\1_{\{1 \}}(x),$$
(other vertices are not involved at this step), with $\widetilde\alpha(t_1)=0$. By using the mass conservation principle we have that (the same factor $2$ in all the terms comes from the fact that the weighted degree $d_x$   is equal to 2 for all $x$), such candidate must verify
$$\left(0-\widetilde\alpha'(t)\right)2+\left(\alpha-\widetilde\alpha'(t) \right)2+\left(0-\widetilde\alpha'(t)\right)2=0,$$
that is $$6 \widetilde\alpha'(t)=2\alpha,$$ that joint the initial condition
$\widetilde\alpha(t_1)=0$ gives us:
$$
\widetilde\alpha(t)=\frac{\alpha}{3}(t-t_1).
$$
Then,  we expect that the solution is given by
\begin{equation}\label{ec.u.2}
\begin{array}{r}u(t,x) = \displaystyle \frac{\alpha}{3}(t-t_1) \1_{\{-1 \}}(x)+ \left(1 + \frac{\alpha}{3}(t-t_1) \right) \1_{\{0 \}}(x) + \frac{\alpha}{3}(t-t_1) \1_{\{ 1 \}}(x),
\end{array}
\end{equation}
for $t_1 \leq t < t_2 = t_1+ \frac{3}{\alpha}=\frac{4}{\alpha}$. Observe that   $t_2$ is the first time when $u(t,x) = 2$ in
$x =0$;  and at this time $u(t,x) = 1$ in
$x =\pm1$ (so $u$ belongs to $K^G_{\infty}$).
Let us now check \eqref{def.solution.3}.
Using the explicit formula for $u(t,x)$ given in \eqref{ec.u.2},
we obtain
$$
\begin{array}{l}
\displaystyle \frac12 \int_{\Z} (f(t,x) - u_t (t,x)) (v(x) - u(t,x))\, d\nu_G(x)  \\[10pt]
\displaystyle =  \left(\alpha -\frac{\alpha}{3} \right)\left(v(0) - \left(1 + \frac{\alpha}{3}(t-t_1) \right) \right)
 -\frac{\alpha}{3} \left(v(1) - \frac{\alpha}{3}(t-t_1) \right)   -\frac{\alpha}{3} \left(v(-1) - \frac{\alpha}{3}(t-t_1) \right)  \\[10pt]
 \displaystyle = \frac{\alpha}{3}\big(    v(0) - v(1) + v(0) - v(-1)    - 2\big) \leq 0
 \end{array}
$$
for $v \in K^G_{\infty}$. Then \eqref{def.solution.3} holds.

$\bullet$ Following similar arguments, for times greater than $t_2$, we get that  the solution is given by
$$
\begin{array}{r}u(t,x) = \displaystyle \left(2 + \frac{\alpha}{5}(t-t_2) \right) \1_{\{0 \}}(x) + \left(1 + \frac{\alpha}{5}(t-t_2) \right) \1_{\{\pm1 \}}(x) +  \frac{\alpha}{5}(t-t_2) \1_{\{\pm2 \}}(x),$$
\end{array}
$$
 for $t_2 \leq t < t_3= t_2 +\frac{5}{\alpha}=\frac{9}{\alpha}.$

$\bullet$ It is easy to generalize and verify the following general
formula that describes the solution for every $t\ge t_1$. For any
given integer  $n\ge 1$  we have
$$
\begin{array}{r}
u(t,x) = \displaystyle\left(n + \frac{\alpha}{2n+1}(t-t_n) \right) \1_{\{0 \}}(x) + \sum_{k=1}^n  \left(n-k + \frac{\alpha}{2n+1}(t-t_n) \right) \1_{\{\pm k \}}(x),
\end{array}
$$
for $t_n= \frac{n^2}{\alpha} \leq t <  t_{n+1}=t_n+ \frac{2n+1}{\alpha}= \frac{(n+1)^2}{\alpha}.$
}
\end{example}

\begin{remark}\rm
    We can prove that the function $u(t, \cdot)$ given by ~\eqref{ec.u.2} is a solution by means of the   mass transport interpretation (Theorem \ref{TrabspSol}).

  Let $T: \Z \rightarrow \Z$ be the map $T(-1)=T(1)= 0$, $T(x)=x$ for all $x \in \Z$, $x \not= -1, 1$. It easy to see that $T$ is a transport map that pushes $u_t(t,.)$ to $f(t,.)$. Moreover, we have
$$\int_{\Z} (f(t,x) - u_t (t,x))u(t,x))\, d\nu_G(x)=\frac{4\alpha}{3}=\int_{\Z}u_t(t,x)d_G(x,T(x))\, d\nu_G(x).
$$
 Then, from \eqref{transesta},  $u(t, \cdot)$ is a Kantorovich potential for $d_G$  between  the source $f(t,\cdot)$ and $u_t(t,.)$. Since   $u \in W^{1,2}(0,T;L^2(V,\nu_G))$,  $u(t) \in K^G_{\infty}$  for all $t \in ]0,T[$ and $u(0,\cdot) = u_0$, by Theorem \ref{TrabspSol}, we have that $u(t, \cdot)$ is a solution. The same can be done for the other time steps.

Other way for proving this is the following. Observe that, for the previous transport map $T$, we have that
$$|u(t,x)-u(t,T(x))|=1=d_G(x,T(x)) \quad \hbox{for the vertices} \ x \in {\rm supp}(u_t(t, \cdot)).$$
Then, by the Dual Criteria for Optimality (Lemma \ref{CritOpt} ), we have that $u(t, \cdot)$ is a Kantorovich potential for $d_G$  between  the source $f(t,\cdot)$ and $u_t(t,.)$, thus by Theorem~\ref{TrabspSol},  we have that $u(t, \cdot)$ is a solution.
  $\Box$
\end{remark}

 In the above example, all the weights are equal, and so they are the weighted degrees.

 \begin{remark}\rm
 Observe that there exist graphs with the same quantity of vertices and the same edges between vertices,  with different weights on the
vertices (so that they are different graphs) but with the same weighted degrees. For example, the weighted degrees of the vertices of following graphs $G_k$ are the same: $G_k=\mathbb{Z}$ with  only $n\sim (n+1)$ and   $w_{n,n+1}=1/k$   if $n\in 2\mathbb{Z}$ and  $w_{n,n+1}=(2k-1)/k$ if $n\in 2\mathbb{Z}+1$. For these graphs, the growth dynamic for sandpiles described in this section is the same.

 Note also that the edges (the connections between vertices) are evidently important in the dynamics: for the weighted cycle graph $C_4=\{x_1,x_2,x_3,x_4\}$ with weights $w_{x_1x_2}=1$, $w_{x_2x_3}=2$, $w_{x_3x_4}=1$, $w_{x_4x_1}=2$ we have that $u=1\chi_{\{x_1\}}+2\chi_{\{x_2\}}+1\chi_{\{x_3\}}$ is stable, but for the complete graph $K_4=\{x_1,x_2,x_3,x_4\}$ with all the  weights equal to $1$, the same $u$ is not stable (the slope between $x_1$ and $x_4$ is $2$), although both graphs are connected, have the same number of vertices and all the vertices have the same weighted degree.
  $\Box$
\end{remark}

  In the next example we see more clearly how the weighted degrees influence in the dynamics.

\begin{example}\label{Example2}{\rm Let us consider, the weighted  star graph $G=(V,E)$ with $V:= \{x_0,x_1,x_2,x_3 \}$, $E:= \{(x_0,x_1), (x_1,x_2),(x_1,x_3) \}$ and  weights $w_{01}:= w_{x_0x_1} \not=0$, $w_{12}:= w_{x_1x_2} \not=0$, $w_{13}:= w_{x_1x_3} \not=0$ and zero  otherwise. We denote $d_0 := d_{x_0} = w_{01}$, $d_1:= d_{x_1} = w_{01} + w_{12} + w_{13},$ $d_2:= d_{x_2} = w_{12}$ and $d_3:= d_{x_3} = w_{13}$.

We take
as source the function
$$
    f(t,x) = \alpha \1_{\{x_0 \}}(x),
    \quad 0 < \alpha,
$$
and as initial datum
$$
    u_0 (x) = 0.
$$
Let us find the solution by looking at its evolution between
some critical times.

$\bullet$ First, for small times, the solution to
$P_{\infty}^G(u_0,f)$ is given by
$$
u(t,x) =  \alpha t \, \1_{\{x_0 \}}(x), \quad \hbox{for} \ 0=t_0 \leq t < t_1 = \frac{1}{\alpha}.
$$
Remark that $t_1 = \frac{1}{\alpha}$ is the first time when
$u(t,x) = 1$, and  it is immediate that $u(t,\cdot)
\in K^G_{\infty}$ and $u_t(t,x) = f(t,x)$, then
\eqref{def.solution.3} holds.

$\bullet$ For times greater than $t_1$, similarly to the previous example, we look for a solution of the form
$$u(t,x) = \displaystyle    \left(1 + \widetilde\alpha(t) \right) \1_{\{x_0 \}}(x) +  \widetilde\alpha(t)\1_{\{x_1 \}}(x),$$
with $\widetilde\alpha(t_1)=0$.   By  the mass conservation principle:
$$(\alpha- \widetilde\alpha'(t))  d_0+(0-\widetilde\alpha'(t))d_1=0.$$
Therefore
$$(d_0+d_1)\widetilde\alpha'(t)=d_0\alpha, $$
that, for $\widetilde\alpha(t_1)=0$, has as solution
$$\widetilde\alpha(t)= \frac{\alpha d_0}{d_0+d_1}(t-t_1).$$
Now it is easy to check   that the sandpile growing solution is given by
the function
\begin{equation}\label{ec.u.2N}
u(t,x) = \left(1 + \frac{\alpha d_0}{d_0+d_1}(t - t_1) \right) \1_{\{x_0 \}}(x) + \frac{\alpha d_0}{d_0+d_1}(t - t_1)  \1_{\{x_1 \}}(x)
\end{equation}
for $ t_1 \leq t < t_2 = t_1 + \frac{1}{\frac{\alpha d_0}{d_0+d_1}}$.
Observe that $t_2  = \frac{1}{\alpha} + \frac{d_0+d_1}{\alpha d_0}$ is the time   when $u(t,x_0) = 2$ and $u(t,x_1)=1$.

$\bullet$ For times greater than $t_2$, working similarly,  the solution is given by
\begin{equation}\label{ec.u.3N}
u(t,x) = \left(2 + k(t-t_2) \right) \1_{\{x_0 \}}(x) + \left(1 + k(t-t_2) \right) \1_{\{x_1 \}}(x) $$ $$+ k(t-t_2)  \1_{\{x_2 \}}(x) + k(t-t_2)  \1_{\{x_3 \}}(x)
\end{equation}
for $t_2 \leq t < t_3:=  t_2 + \frac{1}{k}$, where
$$k = \frac{\alpha d_0}{d_0+ d_1+d_2+d_3}.$$

Now, $\displaystyle t_3 = \frac{1}{\alpha} + \frac{d_0+d_1}{\alpha d_0} + \frac{d_0+ d_1+d_2+d_3}{ \alpha d_0}.$

$\bullet$ It is easy to generalize and verify the following general
formula that describes the solution for every $t\ge t_2$. For any
given integer  $n\ge 2$ we have
\begin{equation}\label{ec.u.nN} \begin{array}{ll}
\displaystyle u(t,x) =& \left(n + \frac{\alpha d_0}{d_0+ d_1+d_2+d_3}(t-t_n) \right) \1_{\{x_0 \}}(x) \\[10pt] &+ \left( (n-1)  +  \frac{\alpha d_0}{d_0+ d_1+d_2+d_3}(t- t_n) \right) \1_{\{x_1 \}}(x) \\[10pt]& + \displaystyle \frac{\alpha d_0}{d_0+ d_1+d_2+d_3}(t-t_n)  \1_{\{x_2 \}}(x)
\\[10pt]& +\displaystyle \frac{\alpha d_0}{d_0+ d_1+d_2+d_3}(t-t_n)  \1_{\{x_3 \}}(x)\end{array}
\end{equation}
for $t_n\le t\le t_{n+1}$ where $$ t_n=  \frac{1}{\alpha} + \frac{d_0+d_1}{\alpha d_0} + (n-2) \frac{d_0+ d_1+d_2+d_3}{ \alpha d_0}$$ and $$  t_{n+1}= \frac{1}{\alpha} + \frac{d_0+d_1}{\alpha d_0} + (n-1) \frac{d_0+ d_1+d_2+d_3}{ \alpha d_0} .$$
}
\end{example}

\begin{remark}\rm
 The solution ~\eqref{ec.u.2N} can be also found from the  mass transport interpretation: initially, the  rate of mass $d_0\alpha$ contributed   by the source at point $x_0$  is distributed to a velocity $\frac{d_0\alpha}{d_0+d_1}$ of $u$ at the two vertex involved.
The same can be done for the other time steps. $\Box$
\end{remark}

Let us see now an example with a source in two points.

\begin{example}\label{Example3}{\rm Consider the weighted graph $G=(V,E)$ with $V:= \{x_1,x_2,x_3,x_4\}$, $E:= \{(x_1,x_2), (x_2,x_3),(x_3,x_4) \}$ and  weights $w_{x_1x_2} =w_{x_2x_3} =w_{x_3x_4} = 1$ and zero  otherwise. Then $d_{x_1} = d_{x_4} =1$ and  $d_{x_2} = d_{x_3} =2$.

We take
as source the function
$$
    f(t,x) = \alpha \1_{\{x_2 \}}(x) + \beta \1_{\{x_3 \}}(x) ,
    \quad 0 < \beta  < \alpha,
$$
and as initial datum
$$
    u_0 (x) = 0.
$$
Let us find the solution by looking at its evolution between
some critical times.

$\bullet$ First, for small times, the solution to
$P_{\infty}^G(u_0,f)$ is given by
$$
u(t,x) =  \alpha t \, \1_{\{x_2 \}}(x) + \beta t \, \1_{\{x_3 \}}(x), \quad \hbox{for} \ 0=t_0 \leq t \le t_1 = \frac{1}{\alpha}.
$$
Remark that $t_1 = \frac{1}{\alpha}$ is the first time when
$u(t,x_2) = 1$ and $u(t_1,x_3) = \frac{\beta}{\alpha} <1$. It is immediate that $u(t,\cdot)
\in K^G_{\infty}$ and $u_t(t,x) = f(t,x)$,  then
\eqref{def.solution.3} holds.

$\bullet$ For times greater than $t_1$,
working as in the previous examples, and attending to the fact that the solution must belong to $K^G_{\infty}$  (slope constraint condition),  we look for a solution of the form
$$
u(t,x) = \widetilde\alpha(t) \1_{\{x_1 \}}(x) + \left(1 + \widetilde\alpha(t) \right) \1_{\{x_2 \}}(x) + \left(\frac{\beta}{\alpha} + \widetilde\beta(t) \right)  \1_{\{x_3 \}}(x),
$$
with $\widetilde\alpha(t_1)=0$ and $\widetilde\beta(t_1)=0$   (at the beginning, the source in $x_2$ does not involve any action in $x_3$ because of the  slope constraint condition, and  in the same way, the source in $x_3$ does not involve any action on $x_1$).  By  the mass conservation principle,
$$ \left(0- \widetilde\alpha'(t)\right)+\left(\alpha-\widetilde\alpha'(t)\right)2=0$$
and $$\left(\beta-\widetilde\beta'(t)\right)2 =0.$$
Hence, using the initial conditions in $t_1$,
$$\widetilde\alpha(t)=\frac{2\alpha}{3}(t -t_1)$$ and
$$\widetilde\beta(t)=\beta(t - t_1).$$
And therefore,
\begin{equation}\label{ec.u.2N3}
u(t,x) = \frac{2\alpha}{3}(t -t_1) \1_{\{x_1 \}}(x) + \left(1 + \frac{2\alpha}{3}(t - t_1) \right) \1_{\{x_2 \}}(x) + \left(\frac{\beta}{\alpha} + \beta(t - t_1) \right)  \1_{\{x_3 \}}(x).
\end{equation}
Observe that
$$\displaystyle  \int_{V} (f(t,x) - u_t (t,x)) (v(x) - u(t,x))\, d\nu_G(x)$$ $$ = - \frac{2\alpha}{3} \left(v(x_1)-  \frac{2\alpha}{3}(t -t_1)\right) + 2\left( \alpha - \frac{2\alpha}{3} \right)\left(v(x_2) - \left(1 + \frac{2\alpha}{3}(t - t_1) \right) \right)  $$ $$= \frac{2\alpha}{3} \left( v(x_2) - v(x_1) - 1\right) \leq 0  $$
for $v \in K^G_{\infty}$,  so~\eqref{def.solution.3} is true. Let us see up which time we have that the slope constraint condition is true.

 We have that
 $u(t, \cdot) \in K^G_{\infty}$ is true if
 \begin{equation}\label{CompCond1or}
 \left(1 + \frac{2\alpha}{3}(t - t_1) \right) - \left(\frac{\beta}{\alpha} + \beta(t - t_1)\right) \leq 1,
 \end{equation}
which is equivalent to
\begin{equation}\label{CompCond1}
\left(t - \frac{1}{\alpha}\right) \left(\frac{2\alpha}{3}- \beta \right) \leq \frac{\beta}{\alpha}.
\end{equation}
Now, \eqref{CompCond1} holds for any  $t > t_1=\frac{1}{\alpha}$ if $\alpha \leq \frac{3\beta}{2}$, and, in the case $\alpha > \frac{3\beta}{2}$, we have that \eqref{CompCond1} holds for
$$\frac{1}{\alpha} < t <  \frac{1}{\alpha}+\frac{\frac{\beta}{\alpha}}{\frac{2\alpha}{3}-\beta}=\frac{2}{2\alpha - 3 \beta}=:t_{comp} .$$
On the other hand, note that   $u(\frac{5}{2\alpha},x_2) = 2$  and   $u(\frac{1}{\beta},x_3) = 1$.  Then we have to compare this two times (from which new vertices are involved in the dynamics), $\frac{5}{2\alpha}$ and $\frac{1}{\beta}$, taking also into account the time $t_{comp}$:

 1. In the case $\frac{1}{\beta}\le \frac{5}{2\alpha}$, that is $\alpha\le \frac{5\beta}{2}$, we have that~\eqref{ec.u.2N3} is true for
$$t_1\le t\le t_2:=\frac{1}{\beta}.$$ This is clear since we begin with the solution  satisfying the slope constraint condition at time $t_1$, and in $x_3$ we get the height $1$ before, or at the same time, in which $x_2$ is attained the height 2; anyway we   make the computations: we always have  the slope   condition~\eqref{CompCond1or} satisfied if $\alpha\le \frac{3\beta}{2}$, and for
$\frac{3\beta}{2}<\alpha\le\frac{5\beta}{2}$, we have that
$$\hbox{min}\left\{\frac{1}{\beta},t_{comp}\right\}=\hbox{min}\left\{\frac{1}{\beta},\frac{2}{2\alpha-3\beta}\right\}=\frac{1}{\beta}.$$
In this time we have that at $x_1$ is $\frac{2(\alpha - \beta)}{3\beta}\le 1$, at $x_2$ is $1+\frac{2(\alpha - \beta)}{3\beta}\le 2$, at $x_3$ is $1$ and at $x_4$ is $0$.

2. In the case $\frac{1}{\beta}> \frac{5}{2\alpha}$, that is $\alpha> \frac{5\beta}{2}\hbox{(}\ge \frac{3\beta}{2}\hbox{)}$, we have that
that~\eqref{ec.u.2N3} is true for
$$t_1\le t\le t_2:=\hbox{min}\left\{\frac{5}{2\alpha},t_{comp}\right\}=\hbox{min}\left\{\frac{5}{2\alpha},\frac{2}{2\alpha-3\beta}\right\}=\frac{2}{2\alpha-3\beta}.$$
In this time we have that at $x_1$ and $x_3$ the height of the sand pile is $$\frac{2\beta}{2\alpha-3\beta}<1,$$ and the height at $x_2$ is
$1+\frac{2\beta}{2\alpha-3\beta}$.

Consequently, from the above discussion we have that the function given in \eqref{ec.u.2N3} is the solution of $P_{\infty}^G(u_0,f)$:
$$ \hbox{for $\frac{1}{\alpha} \leq t \le t_2 = \frac{1}{\beta}$ \ if \ $\alpha\le\frac{5\beta}{2}$, }$$
and
$$ \hbox{for $\frac{1}{\alpha} \leq t \le t_2 = \frac{2}{2\alpha-3\beta}$ \ if \ $\alpha >\frac{5\beta}{2}$.}$$

\indent
1. (continuation): In the case $\alpha\le\frac{5\beta}{2}$, the solution after time $t_2=\frac{1}{\beta}$ is the following, up to a time $t_3$ to be determined (the arguments to arrive to this solution are similar to previous one):
\begin{equation}\label{ec.u.2N4}
\begin{array}{ll}
u(t,x) = \left( \frac{2(\alpha - \beta)}{3\beta}+\frac{2\alpha}{3}(t -t_2) \right) \1_{\{x_1 \}}(x) + \left(1 +\frac{2(\alpha - \beta)}{3\beta} + \frac{2\alpha}{3}(t - t_2) \right) \1_{\{x_2 \}}(x) \\[10pt]+ \left(1+ \frac{2\beta}{3} (t - t_2) \right)  \1_{\{x_3 \}}(x) + \frac{2\beta}{3} (t - t_2) \1_{\{x_4 \}}(x). \end{array}
\end{equation}

Now in this case the  slope constraint condition is
$$  \frac{2(\alpha - \beta)}{3\beta} + \frac{2\alpha}{3}(t - t_2)- \frac{2\beta}{3}(t - t_2) \leq 1.$$
Hence, it is true up to the time
$$ t_{comp}= \frac{3}{2(\alpha - \beta)}.$$
 At this time we have that the solution is a pyramid for all the four vertices:
$$u(t_{comp},x_1)=\frac{2(\alpha - \beta)}{3\beta}+\frac{2\alpha}{3}(t_{comp} -\frac{1}{\beta})=\frac{\alpha+2\beta}{3(\alpha-\beta)}=a+1,$$
$$u(t_{comp}, x_2)  =a+2,$$
$$u(t_{comp},x_3)  =a+1,$$
$$u(t_{comp},x_4)  =a=\frac{5\beta-2\alpha}{3(\alpha-\beta)}\ge 0.$$
Hence   the function given by \eqref{ec.u.2N4} is the solution of $P_{\infty}^G(u_0,f)$ for $t_2 \leq t \le t_3$ with
$$t_3:= t_{comp}= \frac{3}{2(\alpha - \beta)}.$$
 Observe that $t_3=t_2$ if $\alpha=\frac{5\beta}{2}$ (so the above calculations  are only necessary for $\alpha<\frac{5\beta}{2}$). Observe also that if $\alpha=\frac{5\beta}{2}$ then $a=0$,
 but  for $\alpha\approx\beta$ this  first pyramid has a very large height and it is achieved at a very large time.

 Now, for $t > t_3$,   the sandpile (the pyramid we have got) grows up at the same velocity in all points:
$$
\begin{array}{ll}\displaystyle u(t,x) = \left(a+1 + \frac{\alpha + \beta}{3}(t - t_3) \right) \1_{\{x_1 \}}(x) + \left(a+2 + \frac{\alpha + \beta}{3}(t - t_3) \right) \1_{\{x_2 \}}(x) \\[10pt]\displaystyle + \left(a+1 + \frac{\alpha + \beta}{3}(t - t_3) \right) \1_{\{x_3 \}}(x) + \left(a + \frac{\alpha + \beta}{3}(t - t_3) \right) \1_{\{x_4 \}}(x). \end{array}
$$

2. (continuation): In the case $\alpha > \frac{5\beta}{2}$, we have
$$u(t_1,x_1) = u(t_1,x_3) =  \frac{2\beta}{2\alpha -3\beta} <1 \quad \hbox{and} \quad u(t_1,x_2) = 1 + \frac{2\beta}{2\alpha -3\beta}.$$
 This is a pyramid (in these three vertices, without taking into account $x_4$), and similarly to the previous case we have that the following  function will be a solution up to the time in which this pyramid gets the height $1$ at $x_3$:
$$
\begin{array}{ll}
u(t,x) = & \displaystyle  \left( \frac{2\beta}{2\alpha -3\beta}+\frac{2(\alpha + \beta)}{5}(t -t_2) \right) \1_{\{x_1 \}}(x)
\\[10pt]
& \displaystyle + \left(1 +\frac{2\beta}{2\alpha -3\beta} + \frac{2(\alpha + \beta)}{5}(t - t_2) \right) \1_{\{x_2 \}}(x) \\[10pt]
& \displaystyle   + \left(\frac{2\beta}{2\alpha -3\beta}+\frac{2(\alpha + \beta)}{5}(t -t_2) \right)  \1_{\{x_3 \}}(x). \end{array}
$$
Since
 $$1 = u(t,x_3) = \frac{2\beta}{2\alpha -3\beta}+\frac{2(\alpha + \beta)}{5}\left(t -\frac{2}{2\alpha-3\beta}\right),$$
if
 $$t = \frac{7}{2 (\alpha + \beta)},$$
we have that such $u$   is the solution of $P_{\infty}^G(u_0,f)$ for
$$t_2 = \frac{2}{2\alpha-3\beta} \leq t \le t_3:= \frac{7}{2 (\alpha + \beta)}.$$
Now,  at this time $t_3$ we have a pyramid for all the four vertices:
$$u(t_3,x) = \1_{\{x_1 \}}(x) + 2 \1_{\{x_2 \}}(x)+ \1_{\{x_3 \}}(x)+0\1_{\{x_4 \}}.$$
Then,  as we have previously calculated,  the pyramid-function
$$
\begin{array}{ll}
u(t,x) = \left(1 +\frac{\alpha + \beta}{3}(t -t_3) \right) \1_{\{x_1 \}}(x) + \left(2 + \frac{\alpha + \beta}{3}(t - t_3) \right) \1_{\{x_2 \}}(x) \\[10pt]+ \left(1+\frac{\alpha + \beta}{3}(t -t_3) \right)  \1_{\{x_3 \}}(x) + \frac{\alpha + \beta}{3}(t -t_3)\1_{\{x_3 \}}(x) , \end{array}
$$
is the solution of $P_{\infty}^G(u_0,f)$ for $t \geq t_3$.

{\it In summary, a common conclusion in both cases holds for this  graph:} there exists a time $t_m$ (depending on a relation between the sources) for which there is $a \ge 0$ such that  a pyramid of height $a+2$ is achieved for all the four vertices:
$$u(x_1,t_m)=a+1,\ u(x_2,t_m)=a+2,\ u(x_3,t_m)=a+1,\ u(x_4,t_m)=a  .$$
And from this time  the solution of $P_{\infty}^G(u_0,f)$ is this pyramid growing up at velocity $\frac{\alpha + \beta}{3}$:
$$
\begin{array}{ll}
u(t,x) = \left(a +1+\frac{\alpha + \beta}{3}(t -t_m) \right) \1_{\{x_1 \}}(x) + \left(a+2 + \frac{\alpha + \beta}{3}(t - t_m) \right) \1_{\{x_2 \}}(x) \\[10pt]+ \left( a +1 + \frac{\alpha + \beta}{3}(t -t_m) \right)  \1_{\{x_3 \}}(x) + \left(a+\frac{\alpha + \beta}{3}(t -t_m)\right)\1_{\{x_4 \}}(x) , \end{array}
$$
 for all $t \geq t_m$. Concretely, this special time is

$\bullet$ $t_m= \frac{3}{2(\alpha - \beta)} (\ge \frac{1}{\beta})$ if $\alpha\le \frac{5\beta}{2}$, and

$\bullet$ $t_m= \frac{7}{2(\alpha + \beta)}(\le \frac{1}{\beta})$ if $\alpha \ge  \frac{5\beta}{2}$;

$\bullet$ $t_m=\frac{1}{\beta}$ if $\alpha=\frac{5\beta}{2}$.

}\end{example}

We now consider  some examples of  collapse of a datum   $0\le u_0\in L^2(V,\nu_G)\cap L^1(V,\nu_G)$ such that $$
1< L = \Vert \nabla u_0 \Vert_{L^\infty( \nu_G \otimes m^G_x)}.
$$
By Theorem \ref{teo.collapsing.intro} we have that there exists the limit
\begin{equation}\label{limcolapso}
\lim_{p\to \infty } u_p (t,x) = u_\infty (x) \qquad \mbox{ in }
L^2(V,\nu_G),
\end{equation}
which is a function independent of $t$ such that $u_\infty \in K^G_{\infty}$.  Moreover,
$u_\infty(x) = v(1,x)$, where $v$ is the unique strong solution of
the evolution equation
\begin{equation}\label{EcAbal}
\left\{ \begin{array}{ll}
\displaystyle \frac{v}{t} - v_{t} \in \partial
I_{K^G_{\infty}}(v), \quad & t \in ]\tau,1],
\\[10pt]
v(\tau,x)= \tau u_0(x),
\end{array}\right.
\end{equation}
with $\tau = L^{-1}$.  We will obtain $u_\infty$ by solving the above problem.    Observe that, from Theorem~\ref{teo.collapsing.intro} and
 Proposition~\ref{vpositiva}, we have that $\tau u_0\le v(t)\le u_\infty$, and,  by Proposition~\ref{remL1}, $u_\infty\in L^1(V,\nu_G)$, hence we have that $v\in L^2(]\tau,1[;L^2(V,\nu_G)\cap L^1(V,\nu_G))$. Therefore, by  Theorem~\ref{consmass}, we have the following  conservation of  mass  principle for  the above problem: \begin{equation}\label{ConserMass12}
 \int_V v_t(t,x) d \nu_G(x) = \int_V \frac{v(t,x)}{t} d \nu_G(x).
 \end{equation}

\begin{example}\label{ej23} {\rm Consider  the weighted graph $G=(V,E)$ with   $V:= \{x_1,x_2,x_3,x_4 \}$, $E:= \{(x_1,x_2), (x_2,x_3),(x_3,x_4) \}$ and  weights $w_{x_ix_{i+1}} = 1$, $i=1,2,3$ and  $w_{x_ix_j} =0$, otherwise.

Let the initial data be $$u_0(x)= 3 \1_{\{x_2 \}}(x) +    b\1_{\{x_4 \}}(x),$$  with $0\le b\le 9/5$.  For this datum we have that
$$L=\Vert \nabla u_0 \Vert_{L^\infty( \nu_G \otimes m^G_x)}=3.$$
We look for a solution of \eqref{EcAbal},  with initial datum at $\tau=\frac{1}{L}=\frac{1}{3}$ equal to $$v\left(\frac13 \right) = \frac13 u_0(x) = \1_{\{x_2 \}}(x) +  \frac{b}{3} \1_{\{x_4 \}}(x), $$ of the form
$$v(t,x):=  \alpha(t) \1_{\{x_1 \}}(x) + (1+ \alpha(t)) \1_{\{x_2 \}}(x) + \alpha(t) \1_{\{x_3 \}}(x) + \beta(t) \1_{\{x_4 \}}(x);$$
hence, with $$\alpha \left(\frac13 \right)= 0  \quad \hbox{and} \quad  \beta\left(\frac13 \right) = \frac{b}{3}.$$
 Remember that the model is a sandpile growth model in which the source is given by $\frac{v}{t}\ge 0$.
 Since the  rescaled initial datum at the three points $x_1,x_2,x_3$ forms a pyramid $0\1_{\{x_1 \}}(x)+1\1_{\{x_2 \}}(x)+0\1_{\{x_3 \}}(x)$ and in the point $x_4$ its value is below $1$, at the beginning of the process, the sources at points $x_1,x_2,x_3$ do not contribute to $x_4$ and the source at $x_4$ do not contributes at the other points,   therefore we can use the mass conservation principle \eqref{ConserMass12}  as follows  (see Remark \ref{remsoporte}). For  the subgraph $\{x_1,x_2,x_3\}$,
$$ 2\left(\frac{1+\alpha(t)}{t} -  \alpha'(t)\right) + 3\left( \frac{\alpha(t)}{t} -\alpha'(t)  \right) =0,$$
  and for the subgraph $\{x_4\}$,
 and
$$\frac{\beta(t)}{t} - \beta'(t) =0.$$ Then, since $\alpha(1/3)=0$, for $t \geq \frac13$,
 $$ \alpha(t) = \frac{6t - 2}{5};$$ and since $\beta(1/3)=b/3$, for $t\geq \frac13$,
 $$ \beta(t) = bt.$$
For these values of $\alpha(t)$ and $\beta(t)$, we have:  on the one hand that $\vert \alpha(t) - \beta(t) \vert \leq1$ for $t\in[\frac13,1]$, since $0\le b\le 9/5$,  and hence $v(t) \in K^G_{\infty}$; and, on the other hand,
$$\displaystyle  \int_{V} \left(\frac{v(t,x)}{t} - u_t (t,x)\right) (w(x) - u(t,x))\, d\nu_G(x) $$ $$=
\left( \alpha'(t) -\frac{\alpha(t)}{t} \right)(w((x_2) - w(x_1) -1)+ 2\left( \alpha'(t) -\frac{\alpha(t)}{t} \right)(w((x_3) - w(x_2) -1)$$ $$=
  \frac{2}{5t}  (w((x_2) - w(x_1) -1)+   \frac{4}{5t} (w(x_3) - w(x_2) -1) \leq 0$$
for any $w \in K^G_{\infty}$. Consequently, the function
$$v(t,x)=   \frac{6t - 2}{5}   \1_{\{x_1 \}}(x) +  \frac{6t +3}{5}   \1_{\{x_2 \}}(x) +\frac{6t - 2}{5}   \1_{\{x_3 \}}(x) + bt \1_{\{x_4 \}}(x)$$
 is a solution of \eqref{EcAbal}.

 Therefore, we have that  the initial datum $u_0(x)= 3 \1_{\{x_2 \}}(x) +    b\1_{\{x_4 \}}(x)$ collapses to
 $$u_\infty(x) = v(1,x) =  \frac{4}{5}  \1_{\{x_1 \}}(x) +  \frac{9}{5}  \1_{\{x_2 \}}(x) + \frac{4}{5}  \1_{\{x_3 \}}(x) +   b\1_{\{x_4 \}}(x).$$
  Observe that, even in the case $1<b\le \frac95$, the sandpile in $x_4$ does not collapse. This means that, in the limit configuration obtained in~\eqref{limcolapso}, there is a {\it  sandpile collapsing} at  the other vertices that {\it instantaneously prevents}  the collapsing at this point. Hence, one can say that this is a very weak model for describing  collapsing of sandpiles, but still, Problem~\eqref{EcAbal}  is an    interesting and simple    self-organizing  mathematical model for the evolution of a datum $u_0$ that violates the slope condition $\Vert \nabla u_0\Vert_{L^\infty( \nu_G \otimes m^G_x)}\le 1$ to a stationary state ($v(1)=u_\infty$) where such condition is attained, in local or nonlocal models, including this on weighted graphs.

 If in this example we take $b=2$, then we have that the function
$$v(t,x)=   \frac{6t - 2}{5}   \1_{\{x_1 \}}(x) +  \frac{6t +3}{5}   \1_{\{x_2 \}}(x) +\frac{6t - 2}{5}   \1_{\{x_3 \}}(x) +  2t \1_{\{x_4 \}}(x)$$
 is a solution of \eqref{EcAbal} up to the time $t=\frac34$; at this  time we have
 $$v\left(\frac34,x_3\right)=\frac12\ \hbox{ and } \ v\left(\frac34,x_4\right)=\frac32.$$
 From such time  up to the limit time $t=1$, the solution of \eqref{EcAbal} is given by
 $$\left(\frac43t-\frac12\right)\1_{\{x_1 \}}(x) +  \left(\frac43t+\frac12\right)  \1_{\{x_2 \}}(x) +\left(\frac43t-\frac12\right)  \1_{\{x_3 \}}(x) + \left(\frac43t+\frac12\right) \1_{\{x_4 \}}(x),$$
 where the growing up velocity is the same in all point.
 The final configuration is
 $$u_\infty(x)=v(1,x)=\frac56\1_{\{x_1 \}}(x) +  \frac{11}{6}  \1_{\{x_2 \}}(x) +\frac56 \1_{\{x_3 \}}(x) + \frac{11}{6} \1_{\{x_4 \}}(x).$$
 In this case  we have collapse in the two points that violate the slope condition.

}
\end{example}

\begin{example} {\rm     Let us consider, the weighted graph   $V:= \{x_1, x_2, x_3, x_4, x_5, x_6 \}$,  such that $x_i\sim x_{i+1}$ with weight $1$, and there is not any other relation (so, $d_{x_1}=d_{x_6}=1$ and $d_{x_i}=2$ for $i=2,3,4,5$) . Take
  the initial datum $$u_0(x)= 3 \1_{\{x_2 \}}(x) +   \frac95\1_{\{x_4 \}}(x)+2\1_{\{x_5 \}}(x),$$
  where
 working as in the previous examples,  we have that the solution of \eqref{EcAbal} is given by
    $$v(t,x)=   \frac{6t - 2}{5}   \1_{\{x_1 \}}(x) +  \frac{6t +3}{5}   \1_{\{x_2 \}}(x) +\frac{6t - 2}{5}   \1_{\{x_3 \}}(x) +
     \frac95t \1_{\{x_4 \}}(x)+2t\1_{\{x_5 \}}(x)$$
    for $\frac13\le t\le\frac12$; and
    $$v(t,x)=   \frac{6t - 2}{5}   \1_{\{x_1 \}}(x) +  \frac{6t +3}{5}   \1_{\{x_2 \}}(x) +\frac{6t - 2}{5}   \1_{\{x_3 \}}(x)+\frac95t \1_{\{x_4 \}}(x) $$
    $$\qquad+\frac{4t+1}{3}\1_{\{x_5 \}}(x)+\frac{4t-2}{3}\1_{\{x_6 \}}(x)$$
    for $\frac12\le t\le1$.
    Hence,
  $$u_\infty(x)=v(1,x)=\frac{4}{5}  \1_{\{x_1 \}}(x) +  \frac{9}{5}  \1_{\{x_2 \}}(x) + \frac{4}{5}  \1_{\{x_3 \}}(x) $$ $$+  \frac95\1_{\{x_4 \}}(x)
  +\frac53\1_{\{x_5\}}(x) + \frac23  \1_{\{x_6 \}}(x).$$
    Note that although being   $u_0(x_4)<u_0(x_5)$, finally   $u_\infty(x_4)>u_\infty(x_5)$. So,  the collapsing of a sandpile may change the location of  {\it peaks} of an initial configuration.

}
\end{example}

\section{A second model of sandpile growth}

In the above model we have  seen that the dynamic depends of the weights through the weighted degree of the vertices. Here we introduce a new model of  sandpile   in which the dynamic depends   explicitly on the weights. We can also arrive to this model, as for the previous one, by taking limits  as $p\to+\infty$  to the solutions of a $p$-Laplacian evolution equation  but with a different   $p$-Laplacian operator.  This other $p$-Laplacian operator is also used in many other problems in the context of weighted graphs, see for example \cite{Elmoatazetal0}, where different type of $p$-Laplacian  type operators on weighted graphs are described.

\subsection{A different $p$-Laplacian evolution problem}\label{p-LaplacianEpesado}

In this section we continue assuming that $[V(G),\mathcal{B},m^G,\nu_G]$ is the reversible random walk space associated with the weighted graph $G = (V(G), E(G)),$ given in Subsection \ref{example.graphs}, and  that $G$ is connected.  We simplify the writing by using $V=V(G)$.  In this section we will assume that
$$\hbox{there exits $M_w$ such that $ w_{xy}\leq M_w$  for all $x,y \in V$.}
$$

Let $p\ge 3$. We define the following {\it weighted } $p$-Laplacian operator  in $G$:
 $$\Delta_p^w u (x) :=   \frac{1}{d_x} \sum_{y\sim x}\left(\sqrt{w_{xy}}\right)^{p-2} \vert \nabla  u(x,y) \vert^{p-2} \nabla  u(x,y)  w_{xy} .$$
 The  {\it integration by  parts formula} for this model reads as follows, for adequate integrable functions:
$$
  \sum_{x\in V} \Delta_p^w u (x) \, v(x) d\nu_G(x) $$
  $$= - \frac12 \sum_{x\in V}\sum_{y\sim x} \left(\sqrt{w_{xy}}\right)^{p-2} \vert \nabla  u(x,y) \vert^{p-2} \nabla  u(x,y)   \nabla v(x,y) w_{xy}
$$
$$= - \frac12 \sum_{x\in V}\sum_{y\sim x} \left(\sqrt{w_{xy}}\right)^{p} \vert \nabla  u(x,y) \vert^{p-2} \nabla  u(x,y)   \nabla v(x,y).$$
 Observe that this operator coincides with the anisotropic graph $p$-Laplacian described in~\cite[Section 3.3]{Elmoatazetal0} up to the scalar $\frac{1}{d_x}$,
$$\Delta_p^a u (x) :=     \sum_{y\sim x}\left(\sqrt{w_{xy}}\right)^{p} \vert \nabla  u(x,y) \vert^{p-2} \nabla  u(x,y)  ,$$
but in fact, using the different Hilbert structures considered in~\cite{Elmoatazetal0} and here  they act in a similar way:
$$
  \sum_{x\in V} \Delta_p^w u (x) \, v(x) d\nu_G(x)=\sum_{x\in V} \Delta_p^a u (x) \, v(x).
$$

 Consider the {\it evolution problem} in  $G =(V,E,w)$:
\begin{equation}\label{NOnLOp.2pesado}
P_p^w(u_0,f)\quad\left\{\begin{array}{l} \displaystyle u_t (t,x) = \Delta^w_p u(t,x) + f(t,x), \\[10pt]
u(0,x)=u_0 (x).
\end{array}\right.
\end{equation}
Like in the previous section,  problem $P_p^w(u_0,f)$   is the gradient flow in $L^2(V,\nu_G)$ associated
to the functional
$$\begin{array}{ll}
J^w_p(u)\!\! &= \left\{\begin{array}{ll}\displaystyle\frac{1}{2p} \sum_{(x,y) \in V \times V} \left(\sqrt{w_{xy}}\right)^{p-2} \vert \nabla  u(x,y) \vert^p \, w_{xy}    &\hbox{ if} \ u \in  L^2(V, \nu_G) \cap L^p(V, \nu_G), \\[10pt] + \infty  &\hbox{ if} \ u \in  L^2(V, \nu_G) \setminus L^p(V, \nu_G), \end{array} \right.
\\
\\
\ & = \left\{\begin{array}{ll}\displaystyle\frac{1}{2p} \sum_{(x,y) \in V \times V} \left(\sqrt{w_{xy}}\right)^{p} \vert \nabla  u(x,y) \vert^p   \quad &\hbox{if} \ u \in  L^2(V, \nu_G) \cap L^p(V, \nu_G), \\[10pt] + \infty  &\hbox{if} \ u \in  L^2(V, \nu_G) \setminus L^p(V, \nu_G). \end{array} \right.
\end{array}
$$
Let us  introduce the operator $\mathcal{B}^w_p$ in $L^2(V, \nu_G)\times  L^2(V, \nu_G)$ defined as
$$(u,v) \in \mathcal{B}^w_p \iff u \in    L^2(V, \nu_G) \cap L^p(V, \nu_G) \ \hbox{and} \ v = -\Delta_p^G u.$$
With a similar proof to the one for Theorem \ref{CharactTh} we have:
\begin{theorem}\label{CharactThN} The operator $\mathcal{B}^w_p = \partial J^w_p$ is $m$-completely accretive in $L^2(V, \nu_G)$ and has dense domain.
\end{theorem}

Since $P_p^w(u_0,f)$ coincides with the abstract Cauchy problem
$$
\left\{ \begin{array}{ll}
u'(t) + \mathcal{B}^w_p(u(t)) \ni f(t) \quad \ t \geq 0, \\[10pt]
u(0) = u_0,
 \end{array} \right.
$$
by the Brezis-Komura theorem (\cite{Brezis}), having in mind Theorem \ref{CharactThN}, we also have the following existence and uniqueness result.

\begin{theorem}\label{ExitUniqN} For any $u_0 \in L^2(V,\nu_G)$ and $f \in L^2(0, T; L^2(V,\nu_G))$ there exists a unique strong $u(t)$ solution of problem $P_p^w(u_0,f)$, that is  $u \in C([0,T]:L^2(V,\nu_G)))\cap W_{\rm loc}^{1,2}(0,T;L^2(V,\nu_G))$, and, for almost all $t \in ]0,T[ $,  $u(t) \in L^2(V, \nu_G)$  and it satisfies $P_p^w(u_0,f)$.
\end{theorem}

\subsection{Limit as $p \to \infty$}\label{sect.covergtoppesado}

Taking limit as $p \to \infty$ to the functional $J^w_p$  we will now get the functional
\begin{equation}\label{e2MazoNNpesado}
J^w_{\infty}(u) =  \left\{ \begin{array}{ll} 0    & \displaystyle
\hbox{if}
\
\  u \in L^2(V, \nu_G), \ \vert \nabla  u(x,y) \vert\leq \frac{1}{\sqrt{w_{xy}}}  \quad \hbox{if} \ x \sim y,
\\[10pt]
+ \infty  & \hbox{in other case},
\end{array} \right.
\end{equation}
which is the indicator function of
$$K^w_{\infty}:= \left\{ u \in L^2(V, \nu_G), \ \vert \nabla u(x,y) \vert\leq \frac{1}{\sqrt{w_{xy}}}  \quad \hbox{if} \ x \sim y\right\}.$$
  Then, the
 limit
 problem will now be
\begin{equation}\label{NOnLOppesado}
P_{\infty}^w(u_0,f)\quad\left\{\begin{array}{l} \displaystyle f(t,
\cdot)- u_t
(t,.) \in \partial I_{K^w_{\infty}} (u(t,.)), \ \ \ \hbox{ a.e.} \ t \in ]0,T[, \\[10pt]
u(0,x)=u_0 (x).
\end{array}\right.
\end{equation}
Since $I_{K^G_{\infty}}$ is convex and lower semicontinuous in $L^2(V, \nu_G)$,  by the Brezis-Komura theorem (\cite{Brezis}),  for every initial data $u_0 \in K^G_{\infty}$,
problem $P_{\infty}^w(u_0,f)$ has a unique strong solution.

 The limit problem $P_{\infty}^w(u_0,f)$ is   the  model~\eqref{NOnLOppesadoiii} for  sandpile growing in weighted graphs described  in the Introduction.
 Note that this model  takes into account the weights on edges, not only the weighted degrees on vertices, in the dynamics.

With a similar proof of Theorem \ref{MoscoConv1}, we have the following result  (we included a brief proof with the necessary changes).

 \begin{theorem} \label{MoscoConv2}   The functionals $J^w_p$
converge to $J^w_\infty$
as $p\to \infty$, in the sense of Mosco in $L^2(V,\nu_G)$.
 \end{theorem}

 \begin{proof}
To prove
$
 {\rm Epi}(J^w_{\infty}) \subset s\hbox{-}\liminf_{p
\to \infty}{\rm Epi}(J^w_{p}),
$
one main change is to get the equivalent to~\eqref{1857}:  Take $u \in K_\infty^w$, and $u_p:= u$, then, we have
$$J^w_p(u_p) =   \frac{1}{2p}  \sum_{(x,y) \in V \times V} \left(\sqrt{w_{xy}}\right)^{p-2}|u(y) - u(x)|^{p}   w_{xy}
 $$
 $$
= \frac{1}{2p}  \sum_{(x,y) \in V \times V} \left(\sqrt{w_{xy}}\right)^{p-2}|u(y)
 - u(x)|^{p-2} |u(y)
 - u(x)|^2 w_{xy} $$
 $$\leq  \frac{1}{2p}  \sum_{(x,y) \in V \times V}  |u(y)
 - u(x)|^2 w_{xy}   \le \frac{2}{p}  \sum_{x \in V  }  |u(x)|^2 d_x \to 0, $$
 as $ p\to \infty$.

To prove that
 $
 w\hbox{-}\limsup_{p \to \infty} {\rm
Epi}(J^w_{p}) \subset {\rm Epi}(J^w_{\infty}),
$
the main change is to get the equivalent to~\eqref{1903}, and this follows from this estimate (for the corresponding sequence $u_{p_j} \rightharpoonup u$):
$$
\begin{array}{l}
\displaystyle \left( \sum_{(x,y) \in V \times V}
\left(\sqrt{w_{xy}}\right)^{q_j}\left|
u_{p_j} (y)
 - u_{p_j} (x) \right|^{q_j}   \right)^{\frac{1}{q_j}}
\\[10pt]
\displaystyle
=\left( \sum_{(x,y) \in V \times V}
\left(\sqrt{w_{xy}}\right)^{p_j/2}\left|
u_{p_j} (y)
 - u_{p_j} (x) \right|^{p_j/2}
\left|
u_{p_j} (y)
 - u_{p_j} (x) \right| \sqrt{w_{xy}}  \right)^{\frac{1}{q_j}}
\\[10pt]
\displaystyle  \leq
\left( \sum_{(x,y) \in V \times V}
\left(\sqrt{w_{xy}}\right)^{p_j}\left|
u_{p_j} (y)
 - u_{p_j} (x) \right|^{p_j}
  \right)^{\frac{1}{2q_j}}
 \left( \sum_{(x,y) \in V \times V}
 \left|
u_{p_j} (y) - u_{p_j} (x) \right|^{2}w_{xy}
  \right)^{\frac{1}{2q_j}}.
 \end{array}
$$
 \end{proof}

As a  consequence of
  Theorem \ref{convergencia1.Mosco} and Theorem \ref{MoscoConv2} we have:

\begin{theorem} \label{convergencia.p.introW}
Let   $T>0$, $f \in L^2(0,T; L^2(V,\nu_G))$, $u_0\in L^2 (V,\nu_G)$ such that
$\vert \nabla  u(x,y) \vert\leq \frac{1}{\sqrt{w_{xy}}}$ if $x \sim y$ and $u_{p}$ the
unique solution of $P_p^{w}(u_0,f)$. Then, if
 $u_{\infty}$ is the unique strong solution to
$P^w_{\infty}(u_0,f)$,
$$
\displaystyle \lim_{p\to \infty}\sup_{t\in [0,T]}\Vert
u_{p}(t,\cdot)-u_{\infty}(t,\cdot)\Vert_{L^2(V,\nu_G)}=0.
$$
\end{theorem}

\subsection{Collapse of the initial condition}\label{Sect.collapsingW}

Working as in the proof of Theorem \ref{teo.collapsing.intro}, we have the following result.

\begin{theorem} \label{teo.collapsing.introW}
Let $u_p$ be the solution to $P_{p}^w(u_0,0)$ with initial condition
$u_0 \in L^2 (V, \nu_G)$ such that
$$
1< L = \Vert \nabla_w u_0 \Vert_{L^\infty( \nu_G \otimes m^G_x)}.
$$
Then, there exists the limit
$$
\lim_{p\to \infty } u_p (t,x) = u_\infty (x) \qquad \mbox{ in }
L^2(V,\nu_G),
$$
which is a function independent of $t$ such that $u_\infty \in K^w_{\infty}$.  Moreover,
$u_\infty(x) = v(1,x)$, where $v$ is the unique strong solution of
the evolution equation
$$
\left\{ \begin{array}{ll}
\displaystyle \frac{v}{t} - v_{t} \in \partial
I_{K^w_{\infty}}(v), \quad & t \in ]\tau,\infty[,
\\[10pt]
v(\tau,x)= \tau u_0(x),
\end{array}\right.
$$
with $\tau = L^{-1}$.
\end{theorem}

\begin{proof}
 In fact, the main change is related to the equivalent fact for~\eqref{closure1}, and this is true since now, for the equivalent set $C$, we have that, for any $u\in L^2(V,\nu_G)$ and $\lambda>0$, $T_\lambda u\in M_w\lambda C$, where $M_w$ is the bound assumed on the weights $w_{xy}$.
\end{proof}

\subsection{Mass transport interpretation}\label{Sect.Mass.TransferW}

The mass transport interpretation of this second model is similar to the given in Subsection \ref{Sect.Mass.Transfer}  but using the metric $d_w$ defined as
$$ d_w(x,y) = \inf\left\{\displaystyle\sum_{i=1}^n \frac{1}{\sqrt{w_{x_{i-1}x_i}}}:\{x_1,x_2,...,x_n\} \hbox{ is a path connecting $x$ and $y$}\right\}
$$

\begin{remark}\label{rem45nui}\rm  Another energy functional  that can  be considered  is
$$\begin{array}{ll}
\widetilde{J^w_p}(u)\!\! &= \left\{\begin{array}{ll}\displaystyle\frac{1}{2p} \sum_{(x,y) \in V \times V} \left(w_{xy}\right)^{p-1} \vert \nabla  u(x,y) \vert^p    \quad &\hbox{if} \ u \in   L^2(V, \nu_G) \cap L^p(G,\nu_G) , \\[10pt] + \infty  &\hbox{if} \ u \in  L^2(V, \nu_G) \setminus L^p(G,\nu_G), \end{array} \right.
\end{array}
$$
which will be  related in the limit to the indicator function of
 $$\widetilde{K^w_{\infty}}:= \left\{ u \in   L^2(V, \nu_G), \ \vert \nabla u(x,y) \vert\leq \frac{1}{w_{xy}}  \quad \hbox{if} \ x \sim y\right\}.$$
The distance involved in this case for the mass transport interpretation is given by
$$\inf\left\{\displaystyle\sum_{i=1}^n \frac{1}{w_{x_{i-1}x_i}}:\{x_1,x_2,...,x_n\} \hbox{ is a path connecting $x$ and $y$}\right\},$$
usually used in the literature. $\Box$
\end{remark}

\subsection{Explicit solutions}\label{sect.explicitpesado}

 Let  us  see with a very simple example the different  dynamics of the two models.

\begin{example}{\rm Let us consider, the weighted graph $G=(V,E)$ with $V:= \{x_1,x_2,x_3 \}$, $E:= \{(x_1,x_2), (x_2,x_3) \}$ and  weights $w_{x_1x_2} =1$ and $w_{x_2x_3} = 4$ and zero  otherwise. We have $d_{x_1} = 1$, $d_{x_2} = 5$ and $d_{x_3} = 4$.

We take
as source the function
$$
    f(t,x) = \alpha \1_{\{x_2 \}}(x),
    \quad 0 < \alpha,
$$
and as initial datum
$$
    u_0 (x) = 0.
$$

 First, we give the solution of problem $P^G_{\infty}(u_0,f)$. Working as in the Examples of Subsection \ref{sect.explicit}, it is easy to see that, for small times, the solution to
$P_{\infty}^G(u_0,f)$ is given by
$$
u(t,x) =  \alpha t \, \1_{\{x_0 \}}(x), \quad \hbox{for} \ 0=t_0 \leq t < t_1 = \frac{1}{\alpha};
$$
and for later times:

 $$
   \begin{array}{ll}
   u(t,x) &= \left((n-1) +\frac{\alpha}{2}(t- t_n)\right) \1_{\{x_1 \}}(x) + \left(n +\frac{\alpha}{2}(t- t_n)\right) \1_{\{x_2 \}}(x) + \\[10pt] &+\left((n-1)+\frac{\alpha}{2}(t- t_n)\right) \1_{\{x_3 \}}(x), \end{array}
$$
   for $t_{n-1} \leq t \leq t_n:= \frac{1}{\alpha} + (n-1) \frac{2}{\alpha},$ for $n=1,2,3,...$.

   \medskip

  Let us now  to find the solution of problem $P^w_{\infty}(u_0,f)$.

    $\bullet$ First, for small times, the solution to
$P_{\infty}^w(u_0,f)$ is given by
$$
u(t,x) =  \alpha t \, \1_{\{x_0 \}}(x), \quad \hbox{for} \ 0=t_0 \leq t < t_1 = \frac{1}{2\alpha}.
$$
 Observe that, just from the beginning the solutions are different because the  slope condition $u\in K^w_{\infty}$ is different. Note that $t_1 = \frac{1}{2\alpha}$ is the first time when
$\displaystyle u(t,x) = \frac12$.

$\bullet$  For times greater than $t_1$ the support of the solution increases but only $x_3$ enters in the equation since $\displaystyle \frac{1}{\sqrt{w_{23}}}=\frac12<1=\frac{1}{\sqrt{w_{12}}}$.   The solution will have the form
$$
u(t,x) = \left(\frac12 + \widetilde\alpha(t)\right)  \1_{\{x_2\}}(x)+  \widetilde\alpha(t)  \1_{\{x_3 \}}(x),
$$
with $\widetilde\alpha(t_1)=0$. By the mass preservation principle we get
$$\widetilde\alpha(t)=\frac{5\alpha}{9}(t - t_1).$$
Hence,
$$
u(t,x) = \left(\frac12 + \frac{5\alpha}{9}(t - t_1)\right)  \1_{\{x_2\}}(x)+  \frac{5\alpha}{9}(t - t_1)  \1_{\{x_3 \}}(x),
$$
that belongs to $K^w_{\infty}$ for $t_1 \leq t < t_2 = t_1 + \frac{1}{2\frac{5\alpha}{9}}$. Note that $t_2$ is the first time when $u(t,x_3) = \frac12$.
We have
  $$\displaystyle  \int_{V} (f(t,x) - u_t (t,x)) (v(x) - u(t,x))\, d\nu_G(x)  =\frac{20\alpha}{9} \left(v(x_2) - v(x_3) - \frac12 \right)\leq 0$$
  for $v \in K^w_{\infty}$.

$\bullet$ For times greater than $t_2$, working similarly, we get that the solution to
$P_{\infty}^w(u_0,f)$ is given by
$$u(t,x) =  \frac{\alpha}{2}(t-t_2)  \1_{\{x_1 \}}(x) + \left(1 + \frac{\alpha}{2}(t-t_2) \right) \1_{\{x_2 \}}(x) + \left(\frac12 + \frac{\alpha}{2}(t-t_2) \right) \1_{\{x_3 \}}(x),$$
 for $t_2 \leq t < t_3:= t_2 + \frac{1}{\alpha}= \frac{12}{5\alpha}$.
 In fact, we can describe the solution at any time:
$$\begin{array}{lll}
u(t,x) = \left((n-2) \frac12 + \frac{\alpha}{2}(t-t_n) \right) \1_{\{x_1 \}}(x) + \left(n \frac12 + \frac{\alpha}{2}(t-t_n) \right) \1_{\{x_2 \}}(x) \\[10pt] + \left((n-1) \frac12 + \frac{\alpha}{2}(t-t_n) \right) \1_{\{x_3 \}}(x), \end{array}
$$
 for $t_{n-1}\le t\le t_n:= t_{n-1} + \frac{1}{\alpha}$ and $n = 2, 3,4,...$.

}
\end{example}

{\bf Acknowledgements.} The authors has been partially supported by:
\\
 Grant PID2022-136589NB-I00 funded by MCIN/AEI/10.13039/501100011033 and FEDER
\\
   Grant RED2022-134784-T funded by
MCIN/AEI/10.13039/501100011033.

\end{document}